\documentclass [intlimits, 12pt]{amsart}

\usepackage{mathtools} 
\usepackage[top=100pt,bottom=60pt,left=96pt,right=92pt]{geometry}
\usepackage[mathscr]{eucal}
\usepackage{amsmath}
\usepackage{amsthm}
\usepackage{amssymb}
\usepackage{amsfonts} 
\usepackage{mathrsfs}
\usepackage{amscd}
\usepackage{enumitem} 

\usepackage[pdftex]{color}
\usepackage{xcolor}
\usepackage{layout}
\usepackage{stmaryrd} 
\usepackage{float} 
\usepackage{longtable}
\usepackage[pdftex]{graphicx}
\usepackage{txfonts} 
\usepackage{setspace} 
\usepackage{subcaption} 
\usepackage{verbatim}
\usepackage{nicefrac}
\usepackage{textcomp}
\usepackage{halloweenmath} 
\usepackage{hyperref} 

\definecolor{linkblue}{rgb}{0,0,.6}
\definecolor{citered}{rgb}{.7,0,0}
\hypersetup{colorlinks =true, linkcolor=linkblue, citecolor = citered, urlcolor=linkblue}

\newtheorem{theorem}{Theorem}[section]
\newtheorem{proposition}[theorem]{Proposition}
\newtheorem{corollary}[theorem]{Corollary}
\newtheorem{lemma}[theorem]{Lemma}

\newtheorem{definition}[theorem]{Definition}
\newtheorem{remark}[theorem]{Remark}
\newtheorem{example}[theorem]{Example}

\theoremstyle{plain}

\newcommand{\purge}[1]{} 

\newcommand{\vsp}{\vspace{2mm}}

\def\epsilon{\varepsilon}
\def\phi{\varphi}

\newcommand{\al}{\alpha}
\newcommand{\be}{\beta}
\newcommand{\ga}{\gamma}

\newcommand{\lam}{\lambda}

\newcommand{\om}{\omega}

\newcommand{\gati}{{\ti{\ga}}}

\newcommand{\pti}{{\ti{p}}}
\newcommand{\qti}{{\ti{q}}}
\newcommand{\rti}{{\ti{r}}}

\newcommand{\Cti}{{\ti{C}}}

\newcommand{\vhat}{{\hat{v}}}

\def\N{{\mathbb N}}

\def\R{{\mathbb R}}

\def\Z{{\mathbb Z}}

\newcommand{\mcE}{\mathcal E}

\newcommand{\mcH}{\mathcal H}
\newcommand{\mcI}{\mathcal I}

\newcommand{\mcM}{\mathcal M}
\newcommand{\mcN}{\mathcal N}

\newcommand{\mcR}{\mathcal R}

\newcommand{\mcMhat}{\widehat{\mathcal M}}
\newcommand{\mcNhat}{\widehat{\mathcal N}}

\newcommand{\msC}{\mathscr C}

\newcommand{\IFF}{\Leftrightarrow}

\newcommand{\ti}{\tilde}
\newcommand{\x}{\times}
\newcommand{\del}{\partial}

\newcommand{\beq}{\begin{equation}}
\newcommand{\eeq}{\end{equation}}
\newcommand{\beqs}{\begin{equation*}}
\newcommand{\eeqs}{\end{equation*}}

\DeclarePairedDelimiter{\abs}{\lvert}{\rvert}

\DeclareMathOperator{\Diff}{Diff}

\DeclareMathOperator{\Fix}{Fix}

\DeclareMathOperator{\Id}{Id}
\DeclareMathOperator{\Img}{Im}

\DeclareMathOperator{\Morph}{Morph}

\DeclareMathOperator{\Obj}{Obj}

\DeclareMathOperator{\Span}{Span}

\DeclareMathOperator{\Symp}{Symp}

\def\slashii#1{\setbox0=\hbox{$#1$}             
\dimen0=\wd0                                 
\setbox1=\hbox{\sl/} \dimen1=\wd1            
\ifdim\dimen0>\dimen1                        
\rlap{\hbox to \dimen0{\hfil\sl/\hfil}}   
#1                                        
\else                                        
\rlap{\hbox to \dimen1{\hfil$#1$\hfil}}   
\hbox{\sl/}                               
\fi}                                         %
\def\slashiii#1{\setbox0=\hbox{$#1$}#1\hskip-\wd0\hbox to\wd0{\hss\sl/\/\hss}}
\newcommand{\refglue}{Theorem \ref{glue}}
\newcommand{\refcut}{Theorem \ref{cut}}

\newcommand{\refcancel}{Lemma \ref{cancel}}

\newcommand{\refexampleNonComplete}{Example \ref{exampleNonComplete}}

\newcommand{\refgeomeDelComplete}{Lemma \ref{geomDelComplete}}

\newcommand{\refexistenceDirectLim}{Proposition \ref{existenceDirectLim}}

\newcommand{\refshortDirectLim}{Proposition \ref{shortDirectLim}}

\newcommand{\refquotientDirSys}{Lemma \ref{quotientDirSys}}

\newcommand{\refdirLimQuotient}{Corollary \ref{dirLimQuotient}}

\newcommand{\refnoChainMap}{Lemma \ref{noChainMap}}

\newcommand{\refcriterionChainCompat}{Lemma \ref{criterionChainCompat}}

\newcommand{\refnewDefLocalFH}{Definition \ref{newDefLocalFH}}

\newcommand{\refpropModSys}{Lemma \ref{propModSys}}

\newcommand{\refPropHomSys}{Proposition \ref{PropHomSys}}

\newcommand{\refthDirectLim}{Theorem \ref{thDirectLim}}

\begin{document}

\title[Homoclinic Floer homology via direct limits]{Homoclinic Floer homology via direct limits}

\author{Sonja Hohloch}

\date{\today}

\begin{abstract}
Assume $M$ to be $\R^2$ or a closed surface of genus $g \geq 1$ and $\om$ a symplectic form on $M$. Let $\phi: M \to M$ be a symplectomorphism with hyperbolic fixed point $x$ and transversely intersecting stable and unstable manifolds $W^s(\phi, x)$ and $ W^u(\phi, x)$. The intersection points $W^s(\phi, x) \cap\ W^u(\phi, x)=:\mcH(\phi, x)$ are called homoclinic points, and the (un)stable manifolds of a symplectomorphism are Lagrangian submanifolds.

For this Lagrangian intersection problem with its wildly oscillating Lagrangian manifolds and infinite number of intersection points, we introduced in earlier works Floer homologies generated by so-called (semi)primary homoclinic points and analysed their dynamical and geometric properties.

In this paper, we significantly generalise these earlier results by first defining a Floer homology generated by finite sets of contractible homoclinic points. These Floer homologies nevertheless still consider rather `local' aspects of $W^s(\phi, x) \cap\ W^u(\phi, x)$ since their generator sets are finite (but the number of contractible homoclinic points is infinite).

To overcome this issue, we construct a direct limit of these `local' homoclinic Floer homologies over suitable index sets. These direct limits accumulate the information gathered by the finitely generated `local' homoclinic Floer homologies.
\end{abstract}

\maketitle


\section{Introduction}
\label{sec:intro}

In the 1960s, V.\ I.\ Arnold conjectured that the number of fixed points of a nondegenerate Hamiltonian diffeomorphism (= a time-1 map of a nonautnomous Hamiltonian flow) on a closed symplectic manifold $(M, \om)$ is greater or equal to the sum over the Betti numbers of this manifold. For the $2n$-dimensional torus, it was proven by Conley and Zehnder in 1983. Floer \cite{floer1, floer2, floer3} turned the fixed point problem into an intersection problem which allowed him to prove it on more general classes of manifolds. More precisely, he considered the fixed points of the Hamiltonian diffeomorphism as intersection points of the graph of the Hamiltonian diffeomorphism with the diagonal in the symplectic manifold $(M \x M, \om \oplus (-\om))$. The diagonal and the graph are Lagrangian submanifolds, i.e., the symplectic form vanishes on them and they have half the dimension of the underlying manifold. The intersection points of the graph and the diagonal can be seen as critical points of the so-called symplectic action functional. Considering this functional as `Morse function', Floer devised some kind of `infinite dimensional Morse theory' for the symplectic action functional. This theory and its generalizations are nowadays called {\em Floer theory} and the associated homology theory is called {\em Floer homology}. Floer theory turned out to be a very powerful tool and gave rise to many other applications in symplectic geometry, dynamical systems, and other fields of mathematics and physics.


\subsection{History and background of homoclinic points}

In order to motivate the construction of Floer homology generated by so-called homoclinic points we first need some notations and background from homoclinic dynamics: Given a manifold $N$ and a diffeomorphism $f\in \Diff(N)$, we call $x \in N$ an {\em $m$-periodic point} if there exists $m \in \N^{\geq 1}$ with $f^m(x)=x$. For $m=1$, such a point $x$ is usually called a {\em fixed point} and the set of fixed points of $f$ is denoted by $\Fix(f)$.
A fixed point $x$ is {\em hyperbolic} if the absolute values of the eigenvalues of the derivative $Df(x)$ of $f$ in $x$ differs from $1$. The {\em stable manifold} of a hyperbolic fixed point $x$ is defined as
$$W^s(f,x):= \{p \in N \mid\lim_{n \to \infty} f^n(p)=x\}$$ 
and the {\em unstable manifold} as 
$$W^u(f,x):= \{p \in N \mid\lim_{n \to -\infty} f^n(p)=x\}.$$
We call the connected components of $W^s(f,x) \backslash \{x\}$ resp.\ $W^u (f,x) \backslash \{x\}$ the {\em branches} of $W^s(f,x)$ resp.\ $W^u(f,x)$. 
The {\em homoclinic points of $x$} are the intersection points of the stable and unstable manifold of $x$. The set of homoclinic points of $x$ is denoted by 
$$\mcH(f,x) := W^s(f,x) \cap  W^u(f,x).$$
In this convention, we consider the fixed point $x$ also as homoclinic point.
In Figure \ref{tangle}, the complicated intersection behaviour of transversely intersecting stable and unstable manifolds is sketched. It is often referred to as `homoclinic tangle'.

\begin{figure}[h]
\begin{center}

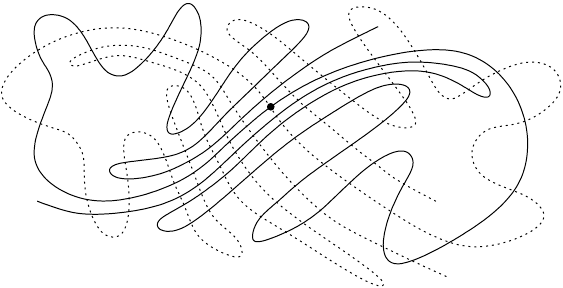

\caption{The intersection behaviour (`homoclinic tangle') of transversely intersecting stable (= continuous line) and unstable (= dotted line) manifold of a hyperbolic fixed point (= bold black dot).}
\label{tangle}

\end{center}
\end{figure}

The {\em orbit} of a point $p \in N$ is given by the set $\{ f^n(p) \mid n \in \Z\}$. Note that the stable and unstable manifolds are invariant under the action
$$\Z \x N \to N, \quad (m, p) \mapsto f^m(p)$$
such that being periodic or homoclinic is in fact a property shared by all points in the orbit.

Intuitively, homoclinic points are the `next more complicated' orbit type after fixed points and periodic points. (Transverse) homoclinic points were discovered by Poincar\'e \cite{poincare1}, \cite{poincare2} around 1890 while studying the $n$-body problem. In 1935, Birkhoff \cite{birkhoff} described the existence of high-periodic points near homoclinic ones, but only Smale's horseshoe in the 1960s explained the implied dynamics precisely. Although homoclinic points have been studied by now by various methods like calculus of variations, perturbation theory, and numerical approximation, many questions are still open.


\subsection{Challenges concerning Floer homology generated by homoclinic points}

If we work with symplectomorphisms instead of diffeomorphisms, the (un)stable manifolds are in fact Lagrangian submanifolds, i.e., the symplectic form vanishes along the (un)stable manifolds and their dimensions equal half of the dimension of the underlying symplectic manifold. This prompts the question if it is possible to construct a Floer homology generated by homoclinic points, i.e., if the intersection points of the Lagrangian (un)stable manifolds admit the construction of Floer homology.

The main problem is the abundance of homoclinic points due to the wild oscillation and accumulation behaviour of the only injectively immersed (un)stable manifolds. This causes tremendous problems in particular for the welldefinedness of the Floer boundary operator which relies, among others, on Fredholm analysis and regularity of solutions of a Cauchy-Riemann type PDE (so-called pseudoholomorphic curves).

If we restrict the dimension of the underlying symplectic manifold to two, we may replace the involved analysis of pseudoholomorphic curves by combinatorics (see de Silva \cite{de silva}, de Silva $\&$ Robbin $\&$ Salamon \cite{de silva-robbin-salamon}, Felshtyn \cite{felshtyn}, Gautschi $\&$ Robbin $\&$ Salamon \cite{gautschi-robbin-salamon}). Note that, in dimension two, being symplectic is the same as being volume preserving w.r.t.\ the symplectic form and, moreover, that any one dimensional submanifold of a two dimensional symplectic manifold is Lagrangian.

While focussing on two dimensional symplectic manifolds avoids all analysis troubles, the problems caused by the infinite number of homoclinic points persist: Homoclinic points are intended to play the role of generators of a chain complex. Thus the infinite number of homoclinic points makes defining chain groups of finite rank a tricky task, not to speak of the welldefinedness of the boundary operator between the chain groups\dots

Now let $(M, \om)$ be a two dimensional symplectic manifold, $\phi: (M, \om) \to (M, \om)$ a symplectomorphism, and $x$ a hyperbolic fixed point of $\phi$. Given $p$, $q \in \mcH:= \mcH(\phi, x)$, denote by $[p,q]_i $ the (one dimensional!) closed segment between $p$ and $q$ in $W^i:=W^i(\phi, x)$ for $i \in \{s,u\}$.
Let $c_p: [0,1] \to W^u \cup W^s$ be a continuous curve with $c_p(0)=x=c_p(1)$ that runs first from $x$ through $[x,p]_u$ to $p$ and then through $[p,x]_s$ back to $x$.
The homotopy class of $p$ is given by $[p]:= [c_p] \in \pi_1(M, x)$. Then
$$\mcH_{[x]}:=\{p \in \mcH \mid [p] =[x] \}$$
is said to be the set of {\em contractible} homoclinic points. It is invariant under the $\Z$-action induced by $\phi$. We define
\begin{align*}
 & \mcH_{pr}:=\mcH_{pr}(\phi, x):= \{ p \in \mcH_{[x]} \mid \ ]p,x[_s \cap\ ]p,x[_u \cap\ \mcH_{[x]} = \emptyset \} , \\
 & \mcH_{s}:=\mcH_{s}(\phi, x):= \{ p \in \mcH_{[x]} \mid \ ]p,x[_s \cap\ ]p,x[_u  = \emptyset \}
\end{align*}
and the elements of $\mcH_{pr}$ are referred to as {\em primary} homoclinic points and those of $\mcH_{s}$ as {\em semiprimary} homoclinic points.
Geometrically, homoclinic points are primary if getting from $p$ to $x$ via the stable and unstable manifold has no  contractible intersection points. They are semiprimary if the ways back do not intersect at all.

Note that $\abs{\mcH_{pr}\slash \Z} < \infty$ and $\abs{\mcH_{s}\slash \Z} < \infty$, i.e., modulo $\Z$-action of $\phi$, the sets of primary and semiprimary points are finite.
This was the motivation to construct Floer homology in dimension two generated by (semi)primary homoclinic points in our earlier work \cite{hohloch1}. Note that these (semi)primary points form a sort of `skeleton' within the intersecting stable and unstable manifolds and the associated (semi)primary homoclinic Floer homologies have interesting links to known quantities in dynamical systems, geometry, and algebra, see Hohloch \cite{hohloch2, hohloch3}.


\subsection{New approach via local homoclinic Floer homology and direct limit}

Assume $M$ to be $\R^2$ or a closed surface of genus $g \geq 1$ and $\om$ a symplectic form on $M$. Let $\phi: (M, \om) \to (M, \om)$ be a symplectomorphism, and $x$ a hyperbolic fixed point of $\phi$ with transversely intersecting stable and unstable manifolds $W^s(\phi, x) \cap\ W^u (\phi, x) =: \mcH(\phi, x)=:\mcH$.

The aim of this paper is to define a Floer homology for significantly larger and more general sets of homoclinic points than the set of (semi)primary homoclinic points used in Hohloch \cite{hohloch1} to define (semi)primary Floer homology.

To construct a Floer homology, we first of all need chain groups to set up the chain complex. The general idea in this paper is to take a finite set of contractible homoclinic points and figure out how to generalize the construction from Hohloch \cite{hohloch1} to this much more general setting.

Consider the set of contractible homoclinic points $\mcH_{[x]}$ (which contains the fixed point $x$). It can be endowed (see Section \ref{secMaslov} for details) with the Maslov grading $\mu: \mcH_{[x]} \to \Z$ and natural convention $\mu(x)=0$.
Set
$$\mcE:= \{E \subset \mcH_{[x]} \mid E \mbox{ finite} \}$$
and let $E \subset \mcE$. Then the Maslov index induces a grading on the elements of $E$ so that the free abelian group $\Z E:= \Span_\Z\{p \mid p \in E\}$ generated by the elements of $E$ can be written as
$$\Z E = \bigoplus_{k \in \Z} \Z E_k$$
with $E_k:=\{ p \in E \mid \mu(p)=k\}$. Since $E$ is finite, only finitely many of the $E_k$ are nontrivial. Hence our candidates for the chain groups of the desired homoclinic Floer chain complex generated by $E$ are
$$
C(E):= \Z E = \bigoplus_{k \in \Z} \Z E_k =:\bigoplus_{k \in \Z} C_k(E).
$$
Now we need to define a boundary operator between the chain groups.
Let $E \in \mcE$ and consider $p, q \in E$ with $\mu(p)-1=\mu(q)$. Classical Lagrangian Floer homology would now count the number of certain pseudoholomorphic strips from $p$ to $q$, but, on two dimensional manifolds, one may replace the analysis by combinatorics (see de Silva \cite{de silva}, de Silva $\&$ Robbin $\&$ Salamon \cite{de silva-robbin-salamon}, Felshtyn \cite{felshtyn}, Gautschi $\&$ Robbin $\&$ Salamon \cite{gautschi-robbin-salamon}) and work with certain `compatible signs' $n(p,q) \in \{-1, 0, +1\}$ instead. For the precise definition of $n(p,q)$ see Section \ref{secSigns} in general and Equation \eqref{signs} in particular.
This suggests to define as candidate for the boundary operator
\begin{equation*}
 \del^E : \Z E \to \Z E, \qquad \del^E p:= \sum_{\stackrel{q \in E}{\mu(q) = \mu(p)-1}} n(p,q) \ q
\end{equation*}
 on the generators and to extend it to $\Z E$ in a linear way. $\del^E$ maps $\Z E_k$ to $\Z E_{k-1}$ and thus gives rise to a family $\del^E_k : C_k(E) \to C_{k-1}(E)$ for all $k \in \Z$. To be a boundary operator, $\del^E$ needs to satisfy $\del^E_{k-1} \circ \del^E_k =0$ for all $k \in \Z$, briefly $\del^E \circ \del^E =0$. Unfortunately there are sets $E \in \mcE$ where this is simply not true, see \refexampleNonComplete, item \ref{notComplete}.

 To overcome this problem, we proceed as follows: a set $E \in \mcE$ is said to be {\em $\del$-complete} if $\del^E \circ \del^E =0$, and otherwise $\del$-incomplete. Now set
 $$
 \overline{\mcE}:= \{E \in \mcE \mid E \ \del\mbox{-complete}\}.
 $$
Then we obtain (for details see Section \ref{sec:localFH}):

\begin{theorem}
\label{th:introFH}
For all $E \in  \overline{\mcE}$, setting
$$
\bigl( C(E), \del^E\bigr):= \bigl( \Z E, \del^ E \bigl)
$$
yields a welldefined chain complex. Its induced homology is called {\em (local) homoclinic Floer homology} of $E$ and denoted by
$$
H(E)= \bigoplus_{k \in \Z} H_k(E) : = \bigoplus_{k \in \Z} \ker \del^E_k \slash \Img \del^E_{k+1}
$$
\end{theorem}

Note that representatives of the generators of (semi)primary homoclinic Floer homology form a $\del$-complete set, so (semi)primary homoclinic Floer homology as defined in Hohloch \cite{hohloch1} is included in the statement above.

There is also a way to assign a Floer homology to $\del$-incomplete sets as proven in Section \ref{sec:pruning}:

\begin{proposition}
Let $E \in \mcE $. Then the pruning algorithm (see Section \ref{sec:pruning}) turns $E$ into a $\del$-complete set $\overline{E} \subseteq E$. If $E$ is already $\del$-complete then $\overline{E} = E$, otherwise $\overline{E} \subset E$ is a strict subset. Thus $H(E):=H(\overline{E})$ is welldefined and coincides for $\del$-complete sets with the definition in Theorem \ref{th:introFH}.
\end{proposition}

Now we want to overcome the still `very local' nature of the Floer homology defined in Theorem \ref{th:introFH} in the sense that we want to see it as one in a sequence of homologies that gather more and more information about the underlying homoclinic tangle. We will do so by defining a suitable index set over which we can take the direct limit of local homoclinic Floer homologies.

For the setting of a direct limit we need a partially ordered index set and a so-called direct system consisting of objects of a category together with transition functions, both indexed by the index set. Moreover, when compatibility with short exact sequences is necessary or wished then the index set should in addition be directed. For precise definitions, we refer the reader to Section \ref{sec:backgroundHomAlg}.

Intuitively, our candidate for the index set should be $\overline{\mcE}$ with partial order induced by the inclusion $\subseteq$ since we would like to take the direct limit over $H(E)$ for all $E \in \overline{\mcE}$. From the inclusions $ D \hookrightarrow E$ for $D, E \in \overline{\mcE}$ with $D \subseteq E$ and their extension to maps $\mcI^{DE}: \Z D \hookrightarrow \Z E$, the transition functions $H( \mcI^{DE}): H(D) \to H(E)$ in this setting would arise.

Unfortunately this does not work without further conditions since $\mcI^{DE}: \Z D \hookrightarrow \Z E$ is an inclusion of graded $\Z$-modules but not necessarily a chain map, i.e., it does not always satisfy $\mcI^{DE} \circ \del^D = \del^E \circ \mcI^{DE}$, see \refnoChainMap.

To overcome this problem, we proceed as follows: a pair $(D, E) $ with $D, E \in \overline{\mcE}$ and $D \subseteq E$ is called {\em chain compatible} if $\mcI^{DE} \circ \del^D = \del^E \circ \mcI^{DE}$. Moreover, a subset $\ti{\mcE} \subset  \overline{\mcE}$ is {\em chain compatible} if, for all $D, E \in \ti{\mcE}$ with $D \subseteq E$, the pair $(D, E)$ is chain compatible.

\begin{proposition}
Let $D, E \in \overline{\mcE}$ with $D \subseteq E$ and let $(D, E)$ be chain compatible.
Then the chain map $\mcI^{DE}: \Z D \to \Z E$ gives rise to a map in homology $H(\mcI^{DE}): H(D) \to H(E)$.
Moreover, for all $ D, E, F \in \overline{\mcE}$ with $D \subseteq E \subseteq F$ and $(D, E)$, $(E, F)$, $(D, F)$ chain compatible, we have
$$H(\mcI^{EF}) \circ H(\mcI^{DE}) = H(\mcI^{DF}).$$
\end{proposition}

This is proven in \refPropHomSys. Now we are ready to pass to the direct limit:

\begin{theorem}
Let $\tilde{\mcE} \subset (\overline{\mcE}, \subseteq)$ be chain compatible. Then, for all $k \in \Z$,
 \begin{align*}
 H^{\ti{\mcE}}_k& := \bigl( \{ H_k(D) \mid D   \in \tilde{\mcE}\} , \{ H(\mcI^{DE}_k) \mid D, E \in \tilde{\mcE}, \ D \subseteq E \} \bigr), \\
 H^{\ti{\mcE}}& := \bigl( \{ H(D) \mid D   \in \tilde{\mcE}\} , \{ H(\mcI^{DE}) \mid D, E \in \tilde{\mcE}, \ D \subseteq E \} \bigr)
\end{align*}
 with (partially ordered) index set $\tilde{\mcE}$ are direct systems of graded $\Z$-modules in the category of $\Z$-moduls. In particular, both are direct systems of abelian groups and therefore admit direct limits.
 If the chain compatible set $\tilde{\mcE}\subset (\overline{\mcE}, \subseteq)$ is not only partially ordered but in fact directed then both direct systems allow passing to the direct limit of short exact sequences and we can assign a welldefined homology via direct limit to $\ti{\mcE}$:
$$
H_k(\ti{\mcE}):= \varinjlim H_k^{\ti{\mcE}} \quad \forall\ k \in \Z \qquad \mbox{and} \qquad H(\ti{\mcE}):= \bigoplus_{k \in \Z} H_k(\ti{\mcE}).
$$
\end{theorem}

This is proven in \refthDirectLim.
Since a direct limit is by construction `quite small' compared to the vastness of the set of homoclinic points there is still much to discover within and around homoclinic tangles!



\subsection{Organization of this paper}
Section \ref{sec:intro} consists of the introduction. Section \ref{sec:FloerStuff} recalls the necessary background from combinatorial Floer homology. Section \ref{sec:localFH} contains the construction and discussion of (local) homoclinic Floer homology for finite sets of contractible homoclinic points. Section \ref{sec:backgroundHomAlg} summarizes necessary notions from homological algebra and category theory. Section \ref{sec:directLimFH} contains the discussion and construction concerning the direct limit of (local) homoclinic Floer homologies.


\subsection*{Acknowledgements}
The author wishes to thank Wendy Lowen and Julia Ramos for helpful explanations and references. The research for this article was partially funded by the
UA BOF DocPro4 project with Antigoon-ID 31722,
the FWO-FNRS EoS project G0H4518N titled {\em Symplectic techniques in differential geometry},
and FWO-FNRS EoS project 40007524 titled {\em Beyond symplectic geometry}.


\section{Maslov index, cutting and gluing, and orientations}
\label{sec:FloerStuff}


Assume $M$ to be $\R^2$ or a closed surface of genus $g \geq 1$ and $\om$ a symplectic form on $M$. Note that, in dimension two, being a symplectic form is the same as being a volume form. Consider $\phi \in \Symp(M, \om)$ with hyperbolic fixed point $x$ and transversely intersecting (un)stable manifolds $W^s:=W^s(\phi, x)$ and $W^u:=W^u(\phi, x)$. This implies in particular that the (un)stable manifolds are {\em one dimensional}. Abbreviate by 
$$\mcH:= \mcH(\phi, x) :=W^s(\phi, x) \cap W^u(\phi, x)$$ 
the set of homoclinic points which includes the fixed point $x$. The symplectomorphism $\phi$ induces the $\Z$-action
$$\Z \x \mcH \to \mcH, \qquad (n, p) \mapsto \phi^n(p)$$ 
on the {\em set of homoclinic points} of $x$. Note that for transversely intersecting $W^s \cap W^u$, the sets $\mcH$ and $\mcH \slash \Z$ are both infinite as a simple glance at Figure \ref{tangle} shows. 


\subsection{Maslov index and grading}
\label{secMaslov}

Given $p$, $q \in \mcH$, denote by $[p,q]_i $ the (one dimensional!) closed segment between $p$ and $q$ in $W^i$ for $i \in \{s,u\}$.
Let $c_p: [0,1] \to W^u \cup W^s$ be a continuous curve with $c_p(0)=x=c_p(1)$ that runs first from $x$ through $[x,p]_u$ to $p$ and then through $[p,x]_s$ back to $x$.
The homotopy class of $p$ is given by $[p]:= [c_p] \in \pi_1(M, x)$. Then $\mcH_{[x]}:=\{p \in \mcH \mid [p] =[x] \}$ is said to be the set of {\em contractible} homoclinic points. It is invariant under the action of $\phi$.

In Hohloch \cite[Section 2.1]{hohloch1}, we showed that there is a (relative) {\em Maslov index} $\mu(p, q) \in \Z$ for $p$, $q \in \mcH$ if $[p] =[q]$. Intuitively, it can be seen as follows:
If we assume the intersections of the Lagrangians to be perpendicular and if we flip $+90^\circ$ at the `vertex' $q$ and $-90^\circ$ at the `vertex' $p$ we can identify $\mu(p,q)$ in our two dimensional setting with twice the winding number of the unit tangent vector of a loop starting in $p$, running through $[p,q]_u$ to $q$ and through $[p,q]_s$ back to $p$. 
Moreover, we have $\mu(p, q)= \mu(\phi^n(p),\phi^n(q))$ for $n \in \Z$ and $p,q \in \mcH_{[x]}$.
For contractible homoclinic points, the (relative) Maslov index yields a natural {\em grading} 
$$\mu: \mcH_{[x]} \to \Z, \qquad  \mu(p):=\mu(p,x)$$ 
where we use the convention $\mu(x)=\mu(x,x)=0$.
Thus we get for contractible homoclinic points $p$ and $q$
\begin{equation*}
\mu(p,q)= \mu(p,x) +\mu(x,q)=\mu(p,x)- \mu(q,x)= \mu(p)-\mu(q).
\end{equation*}
If $p$ and $q$ have the same homotopy class but are not contractible, i.e.\ $[p]=[q] \neq [x]$, then we still can define a grading within this homotopy class by picking a base point $x_0$ with $[x_0]=[p] = [q]$, but there is no `natural' choice for $x_0$. Changing the base point will change the grading by a constant integer. For noncontractible homoclinic points, the relative Maslov index needs not be preserved under iteration.


\subsection{Immersions, di-gons and hearts}

In this section, we recall the definitions and results of Hohloch \cite[Section 2.2]{hohloch1}.

A {\em di-gon} is the polygon $D\subset \R^2$ sketched in Figure \ref{digon} (a), having precisely two vertices, both convex and located at $(-1,
0)$ and $(1,0)$. Denote its upper boundary by
$B_s$ and its lower boundary by $B_u$.
A {\em heart} is either the polygon $D_b $ of Figure \ref{digon} (b) or the
polygon $D_c$ of Figure \ref{digon} (c). It is characterized by two
vertices at $(-1,0)$ and $(1,0)$, one is convex and one concave. Denote
their upper boundaries by $B_s$ and their lower boundaries by $B_u$.

\begin{figure}[h]
\begin{center}

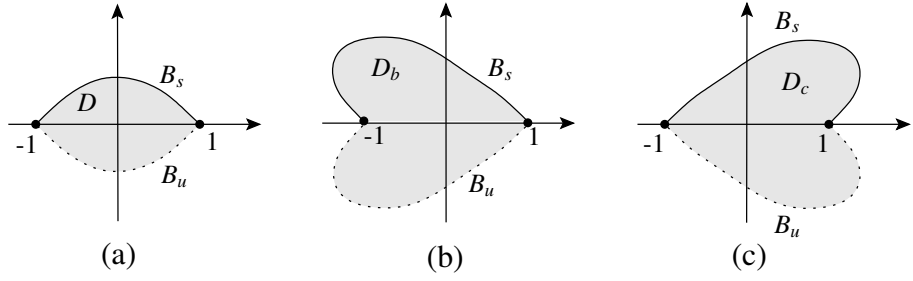

\caption{Di-gon and hearts.}
\label{digon}

\end{center}
\end{figure}

In the following, immersions of di-gons and hearts are to be immersions also on
the boundaries and vertices. This means in particular that the image of a small neighbourhood of a convex (concave) vertex of a
polygon is a wedge-shaped region with angle smaller (larger) than $\pi$.
Let $p$, $q \in \mcH$ with $\mu(p,q)=1$. We define
$\mcM(p,q)$ to be the {\em space of smooth, immersed di-gons} $v : D \to M$
that are orientation preserving and satisfy
$v(B_u) \subset W^u$, $v(B_s) \subset W^s$, 
$v(-1,0)=p$ and $v(1,0)=q$.
Now denote by $G(D)$ the group of orientation preserving diffeomorphisms of $D$
which preserve the vertices. Then we call $\mcMhat(p,q):=\mcM(p,q) \slash G(D)$ the
{\em space of unparametrized immersed di-gons.}

We remark that $\abs{\mcMhat(p,q)} \in \{0,1\}$ for $p$
and $q$ with $\mu(p,q)=1$ since there is exactly one segment $[p,q]_i$, $i \in \{s,u\}$, joining $p$, $q
\in \mcH$ because of $\pi_2(M)=0$.

Now let $p$, $q \in \mcH$ with $\mu(p,q)=2$. We
define $\mcN_b(p,q)$ resp.\ $\mcN_c(p,q)$ to be the {\em space of smooth immersed
hearts} $w: D_b \to M$ resp.\ $w: D_c \to M$ that are
orientation preserving and satisfy
$w(B_u) \subset W^u$, $w(B_s) \subset W^s$, 
$w(-1,0)=p$ and $w(1,0)=q$.
We set $\mcN(p,q):=\mcN_b(p,q) \ \dot{\cup}\ \mcN_c(p,q)$.
Now denote by $G(D_b)$ resp.\ $G(D_c)$ the group of orientation preserving
diffeomorphisms of $D_b$ resp.\ $D_c$ that preserve the vertices. Call
$\mcNhat_b(p,q):=\mcN_b(p,q) \slash G(D_b)$ resp.\ $\mcNhat_c(p,q):=\mcN_c(p,q)
\slash G(D_c)$ and $\mcNhat(p,q) :=\mcNhat_b(p,q) \ \dot{\cup}\ \mcNhat_c(p,q)$
the {\em spaces of unparametrized immersed hearts.}

When working with $\mcM(p,q)$ and $\mcN(p,r)$ we always implicitly
assume $p$, $q$, $r \in \mcH$ with $[p]=[q]$, $[p]=[r]$, $\mu(p,q)=1$ and
$\mu(p,r)=2$.


\subsection{Cutting and gluing}

We now recall some results from Hohloch \cite[Section 2.4]{hohloch1} that are necessary for the construction and welldefinedness of the boundary operator later on.
Intuitively, {\em gluing} of two immersed di-gons $v \in \mcMhat(p,q)$ and $\vhat \in
\mcMhat(q,r)$ with $\mu(p,q)=1=\mu(q,r)$ and thus $\mu(p,r)=2$ is the procedure that recognizes the tupel $(v,\vhat)$ as an element of
$\mcNhat(p,r)$. 
The {\em cutting} procedure, on the other side, is the `inverse' procedure that starts with $w \in \mcNhat(p,r)$ and
locates two points $q_a$, $q_b \in \mcH$ such that $w$ can be idenitified with either a tupel $(v,\vhat) \in \mcMhat(p,q_a) \x \mcMhat(q_a,r)$ or a tupel
$(v', \vhat') \in \mcMhat(p,q_b) \x \mcMhat(q_b,r)$. The geometric intuition of gluing and cutting is sketched in Figure \ref{cutGlue}.

\begin{figure}[h]
\begin{center}

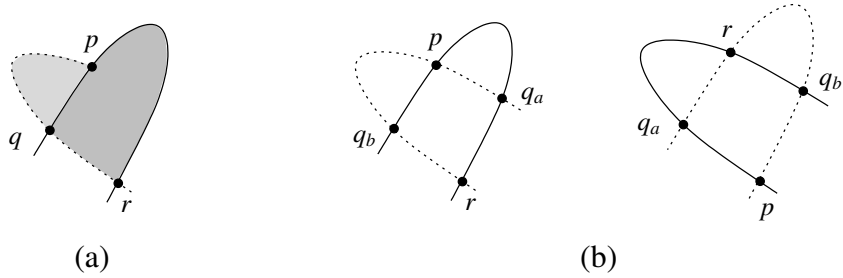

\caption{Geometric intuition for gluing in (a) and cutting with the two placements of the concave and convex vertices in (b). The unstable manifold $W^u$ is drawn with a dotted line and the stable manifold $W^s$ with a full line.}
\label{cutGlue}

\end{center}
\end{figure}

\begin{theorem}[Hohloch \cite{hohloch1}, Theorem 14, `Gluing']
\label{glue}
Let $p$, $q$, $ r \in \mcH$ with $[p]=[q]=[r]$ and $\mu(p,q)=1=\mu(q,r)$. Consider $v
\in \mcMhat(p,q)$ and $\vhat \in \mcMhat(q,r)$. Then the {\em gluing construction
$\#$} for $v $ and $\vhat$ yields an immersed heart $w:=\vhat\#v \in
\mcNhat(p,r)$.
\end{theorem}

Recall that the two connected components of $W^s\backslash\{x\}$ resp.\
$W^u\backslash \{x\}$ are called the {\em branches} of the (un)stable manifolds.
$W^u$ and $W^s$ are said to be {\em strongly intersecting} (w.r.t.\ $x$) if each
branch of $W^u$ intersects each branch of $W^s$, meaning $W^i_+ \cap W_j^+ \neq
\emptyset \neq W^i_- \cap W_j^+ $ for $i, \ j \in \{0,1\}$ and $i \neq j$.
For a discussion when strongly intersecting is a generic property, see Section 2.4 in Hohloch \cite{hohloch1}.

\begin{theorem}[Hohloch \cite{hohloch1}, Theorem 15, `Cutting']
\label{cut}
Let $W^u$ and $W^s$ be strongly intersecting and let all their intersections be transverse.
Consider $p$, $ r \in \mcH$ with $[p]=[r]$ and $\mu(p,r)=2$ and $w \in \mcN(p,r)$. 
Then there are distinct, unique $q_a$, $q_b \in \mcH$ with
$\mu(p,q_i)=1=\mu(q_i, r)$ and $v_i \in \mcM(p,q_i)$, $\vhat_i \in \mcM(q_i,r)$
such that $\vhat_i \#v_i=w$ for $i \in \{a,b\}$. The points $q_a$ and $q_b $ are called {\em cutting partners}.
\end{theorem}

This implies that the map 
$$
(p,r) \mbox{ as in \refcut} \quad \mapsto \quad \mbox{their cutting partners } (q_a, q_b)
$$ 
is welldefined.
Note that this map is not injective since $(q_a, q_b)$ may appear also as cutting partners of a different tuple $(p,r')$, see Figure \ref{sameCutPart}. Moreover, note that the shape of the immersions in $\mcN(p,r)$ and $\mcN(p,r')$ disagree since they are determined by their boundaries $[p,r]_s \cup [p,r]_u$ resp.\ $[p,r']_s \cup [p,r']_u$ which do not coincide.

\begin{figure}[h]
\begin{center}

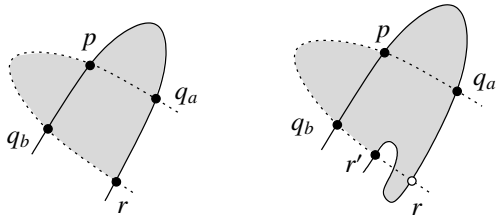

\caption{The points $q_a$ and $q_b$ may appear as cutting partners for two different tuples $(p,r)$ and $(p,r')$ that satisfy both the requirements of \refcut.}
\label{sameCutPart}

\end{center}
\end{figure}


\subsection{Signs and coherent orientations}
\label{secSigns}

In order to define the chain groups over the ring $\Z$ instead of the field $\Z \slash 2 \Z$, we have to show that cutting and gluing are behaving `coherently' when we endow $\mcMhat(p,q)$ with an `orientation', i.e., when we, instead of just counting $\abs{\mcMhat(p,q)}\ mod \ 2$, assign to $\mcMhat(p,q)$ an integer $n(p,q) \in \Z$. We recall the necessary definitions and refer to Hohloch \cite[Section 3.2]{hohloch1} for details.

Now let $\phi$ be a symplectomorphism on a two dimensional manifold with hyperbolic fixed point $x$ and stable and unstable manifolds $W^s$ and $W^u$. Such a $\phi$ is either orientation preserving on $W^s$ {\it and} $W^u$ or orientation reversing on {\em both}. In the first case, $\phi$ is said to be {\em $W$-orientation preserving} and in the latter {\em $W$-orientation reversing}.

Consider $W^u$ and endow it with an orientation denoted by $o_u$. Moreover, for all $p$, $q \in \mcH$ with $\mu(p,q)=1$, endow the segment $[p,q]_u$ with the orientation given by the `direction' from $p$ to $q$ and denote this orientation by $o_{pq}$. If $\phi $ is $W$-orientation preserving, we set
\begin{equation}
\label{signs}
n(p,q):=
\left\{
\begin{aligned}
+1 & \quad \mbox{if } \mcMhat(p,q) \neq \emptyset \mbox{ and } o_{pq}=o_u,  \\
 -1 & \quad \mbox{if } \mcMhat(p,q) \neq \emptyset \mbox{ and } o_{pq} \neq o_u,
 \\
0 & \quad \mbox{if } \mcMhat(p,q)=\emptyset.
\end{aligned}
\right.
\end{equation}
Using $-o_u$ instead of $o_u$ flips the sign of $n(p,q)$. The same happens if we use the orientation $o_{qp}=-o_{pq}$ instead of $o_{pq}$. Note that $n(p,q)$ could also be defined using an orientation on $W^s$. This also results in sign flips, for details we refer to Hohloch \cite[Section 3.2]{hohloch1}. 

\begin{remark}
In case $\phi$ is $W$-orientation reversing, we have to count modulo 2, i.e. using signs of the form $n_2(p,q):= n(p,q) \ mod \ 2$.
\end{remark}

For details we refer again to Hohloch \cite[Section 3.2]{hohloch1}.
Now we explain what we mean with signs being `coherent' w.r.t.\ gluing and cutting:

\begin{lemma}[Hohloch \cite{hohloch1}, Lemma 21]
\label{cancel}
Consider $p$, $r \in \mcH$ with $\mu(p,r)=2$ and $w \in \mcN(p,r)$. For $i \in
\{a, b\}$, let $q_i \in \mcH$ with $\mu(p,q_i)=1=\mu(q_i,r)$ and $\vhat_i \in
\mcM(p,q_i)$ and $v_i \in \mcM(q_i,r)$ such that $\vhat_i \#v_i=w$.
Then
\beqs
n(p,q_a)\cdot n(q_a,r)=- n(p,q_b) \cdot n(q_b,r)
\eeqs
and this relation also is true for $n_2$.
\end{lemma}


\section{Local homoclinic Floer homology}
\label{sec:localFH}

In this section, we aim at defining a Floer homology generated by a finite set of `arbitrary' homoclinic points in order to generalize the type of generators which, in previous versions of homoclinic Floer homology (see Hohloch \cite{hohloch1, hohloch2, hohloch3}), were always required to be (semi)primary.

\vsp

Within this section, let the symplectic manifold $(M, \om)$ be $\R^2$ or a closed surface of genus $g \geq 1$ both endowed with a volume form. Let $\phi: M \to M$ be a $W$-orientation preserving symplectomorphism with hyperbolic fixed point $x$ whose stable and unstable manifolds $W^s:=W^s(\phi, x)$ and $W^u:=W^u(\phi, x)$ are strongly and transversely intersecting.


\subsection{Basic notions and notations}

Consider the set
$$\mcE:=\{ E \subseteq \mcH_{[x]} \mid \abs{E}< \infty \}$$
of finite sets of contractible homoclinic points. For each $E \in \mcE$, there exist
$$k^+:=k^+(E):=\max\{ \  \mu(p) \mid p \in E\} \quad \mbox{and} \quad k^-:=k^-(E):=\min\{ \ \mu(p) \mid p \in E\} .$$
For $E \in \mcE$ and $k \in \Z$, we define
$$
E_k:= \{ p \in E \mid \mu(p) = k\}. 
$$
Therefore, we have
$$
E  = \bigcup_{k \in \Z} E_k= \bigcup_{k^- \leq k \leq k^+} E_k \qquad \mbox{for all } E \in \mcE.
$$
Given $E \in \mcE$, define
$$\Z E:= \Span_\Z \{\ p \mid p \in E \ \}$$
which decomposes into
$$
\Z E   = \bigoplus_{k \in \Z} \Z E_k=\bigoplus_{k^- \leq k \leq k^+} \Z E_k
$$
where we set $\Z E_k =:\{0\}$ whenever $E_k = \emptyset$. Note that, for $D, E \in \mcE$ with $D \subseteq E$, we have also $D_k \subseteq E_k$ for all $k \in \Z$ and thus in fact
$$
\Z D \subseteq \Z E \qquad \mbox{and} \qquad \Z D_k \subseteq \Z E_k.
$$
Hence the grading by the Maslov index is compatible with the structure as freely generated abelian group or $\Z$-module, i.e., we obtained in fact {\em graded} $\Z$-modules.


\subsection{Welldefinedness of the boundary operator}
\label{sec:problemsDelComplete}

If we want to define a Floer homology generated by $E \in \mcE$, then we first need to define a (candidate for a) boundary operator on the (candidate for our) chain complex and then show that it actually is a boundary operator so that we may pass to homology.

Let $\Z E = \bigoplus_{k \in \Z} \Z E_k$ be our chain group generated by $E$. Now define
\begin{equation}
 \label{defDel}
 \del^E : \Z E \to \Z E, \qquad \del^E p:= \sum_{\stackrel{q \in E}{\mu(q) = \mu(p)-1}} n(p,q) \ q
\end{equation}
on the generators and extend $\del^E$ by linearity to $\Z E$. This is the intuitive generalization of the boundary operator of primary homoclinic Floer homology from Hohloch \cite{hohloch1}.
Since $\abs{E}< \infty$ and $n(p,q) \in\{ \pm 1, 0\}$, the sum is finite and welldefined. The map $\del^E$ decomposes in fact as $\del^E=(\del^E_k)_{k \in \Z}$ with
$$
\del^E_k : \Z E_k \to \Z E_{k-1}, \qquad \del^E_k p= \sum_{q \in E_{k-1}} n(p,q) \ q.
$$
Note that $\del^E_k =0$ for all $k \leq k^-(E)$ and all $k > k^+(E)$.
To obtain a chain complex we need the concatenation $\del^E \circ \del^E$ to be welldefined and to yield $\del^E \circ \del ^E =0$, more precisely, we need $\del^E_{k-1} \circ \del^E_{k} =0$ for all $k \in \Z$ so that the sequence below induced by $\Z E = \bigoplus_{k \in \Z} \Z E_k $
$$
\dots \stackrel{\del^E_{k+1}}{\longrightarrow} \Z E_k  \stackrel{\del^E_k}{\longrightarrow} \Z E_{k-1} \stackrel{\del^E_{k-1}}{\longrightarrow} \Z E_{k-2}  \stackrel{\del^E_{k-2}}{\longrightarrow} \dots
$$
give rise to a chain complex with boundary operator $\del^E$.

Note that we work with {\em finite} sets of homoclinic points for the following reason:

\begin{example}
If we admit $E \subseteq \mcH_{[x]}$ with $\abs{E}= \infty$, then $\del^E p$ may be an infinite sum, see the choice of $p$ and some of the (infinitely many) $q$ with $n(p,q) \neq 0$ in Figure \ref{boundaryOperatorInfinite}.
\end{example}

An infinite sum may cause problems in particular for the welldefinedness of the double sum arising from $\del^E \circ \del^E$. We avoid this problem by working with finite sets.

\begin{figure}[h]
\begin{center}

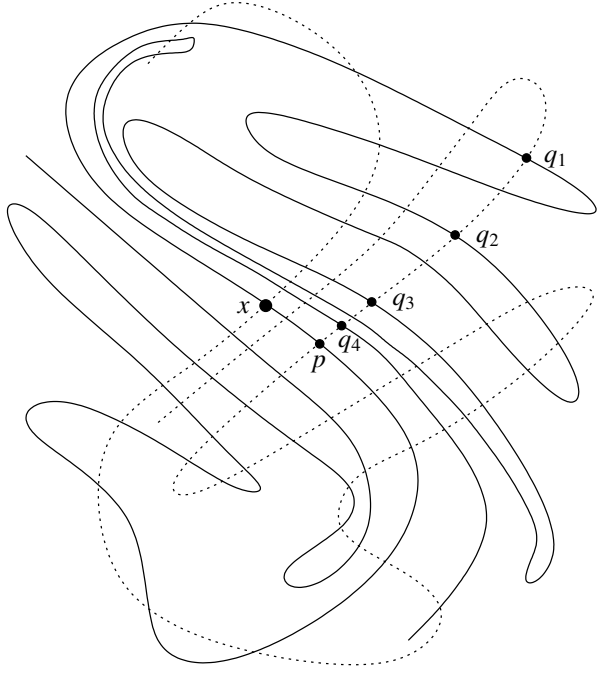

\caption{The unstable manifold is drawn with a dotted line and the stable manifold with a continuous line. Given a homoclinic point $p$ as chosen in the figure, then we find $\abs{ \{q \in \mcH \mid n(p,q) \neq 0\} } = \infty$ since $n(p,q_i) \neq 0$ for at least $q_1, q_2, q_3, q_4, \dots$}

\label{boundaryOperatorInfinite}

\end{center}
\end{figure}

We will now observe that there are sets $E \subset \mcE$ where the above approach leads to a welldefined boundary operator as can be seen in items \ref{completeOne} and \ref{completeTwo} of \refexampleNonComplete.
Unfortunately there are also sets where $\del^E \circ \del^E \neq 0$, see \refexampleNonComplete, item \ref{notComplete}.

\newpage

\begin{example}
\label{exampleNonComplete}
\mbox{ } 
\begin{enumerate}[label=\arabic*)]
  \item \label{completeOne}
 Consider the set $E:=\{p, q_a, q_b, r\}$ in the left part of Figure \ref{cutGlue} (b) and endow $W^u$ with an orientation `pointing' from $q_a$ to $r$. Then we obtain $\del^E p = - q_a + q_b$ and $\del^E q_a = r = \del^E q_b$ and $\del^E r =0$ and thus $(\del^E \circ \del^E) p = \del^E(- q_a + q_b) = -r + r =0$ and $(\del^E \circ \del^E ) q_a = 0$ and $(\del^E \circ \del^E ) q_b = 0$ and $(\del^E \circ \del^E ) r = 0$ implying $\del^E \circ \del^E =0$ on $\Z E$.
 \item \label{completeTwo}
 Considering the right part of Figure \ref{cutGlue} (b), an analogous calculation also shows $\del^E \circ \del^E =0$ on $E:=\{p, q_a, q_b, r\}$.
 \item \label{notComplete}
 Consider the set $E:=\{p, q, r\}$ in Figure \ref{cutGlue} (a) and endow $W^u$ with an orientation `pointing' from $p$ to $r$. Then we find $\del^E p = q$ and $\del^E q = r$ implying $(\del^E\circ \del^E) p = \del^E (\del^E p) =  \del^E (-q) = r \neq 0$. Thus $\del^E \circ \del^E \neq 0$ on $\Z E$.
\end{enumerate}
\end{example}

Thus we either have to exclude all $E \in \mcE$ with $\del^E \circ \del^E \neq 0$ or find a way to assign to $E$ a unique set on which the square of the boundary operator vanishes.

\begin{definition}
 $E \in \mcE$ is said to be {\em $\del$-complete} if $\del^E \circ \del^E =0$ on $\Z E$. Otherwise $E$ is referred to as {\em $\del$-incomplete}.
 We call $\Z E$ {\em $\del$-complete} if the underlying $E \in \mcE$ is $\del$-complete and otherwise {\em $\del$-incomplete}.
 Moreover, we set
 $$
 \overline{\mcE} :=\{ E \in \mcE \mid E \mbox{ is } \del \mbox{-complete}\}.
 $$
\end{definition}

\refexampleNonComplete\ b) and c) show that $\overline{\mcE}  \neq \emptyset$ for strongly and transversely intersecting homoclinic tangles since these `heart shaped' constellations of points appear naturally.
Geometrically, being $\del$-complete means the following.

\begin{lemma}
\label{geomDelComplete}
The following two statements are equivalent:
\begin{enumerate}[label=\arabic*)]
 \item
 $E \in \mcE$ is $\del$-complete.
 \item
 Given $E \in \mcE$, then, for all $k \in \{k^+(E), \dots, k^-(E)+ 2\}$ and all triple $(p,q,r) \in E_k \times E_{k-1} \times E_{k-2}$ which give rise to an immersed heart in the sense of \refglue, the cutting partner of $q$ in the sense of \refcut\ is contained in $E_{k-1}$.
\end{enumerate}
\end{lemma}

\begin{proof}
Let $E \in \mcE$ and $p \in E$. We calculate for $p \in E_k$
\begin{align*}
   \del_{k-1}^E \circ \del_k^E (p)  =  \del^E_{k-1} \left( \sum_{q \in E_{k-1}} n(p,q) \ q  \right)
   =   \sum_{q \in E_{k-1}} \sum_{r \in E_{k-2}} n(p,q) \ n(q,r) \ r .
 \end{align*}
This expression vanishes if and only if for every $q \in E_{k-1}$ with $n(p,q) \ n(q,r)\neq 0$ there exists precisely one $\qti \in E_{k-1}$ with $n(p,\qti) \ n(\qti,r)\neq 0$ and
$$
n(p,q) \ n(q,r) = -n(p,\qti) \ n(\qti,r).
$$
According to \refcancel, this is precisely the case when $q$ and $\qti$ are cutting partners in the sense of \refcut. Since the cutting and gluing operation in \refcut\ and \refglue\ are inverse operations, the cutting partner of a given triple $(p,q,r) \in E_k \times E_{k-1} \times E_{k-2}$ is unique.
\end{proof}


\subsection{Homoclinic Floer homology of $\del$-complete sets}

Now we associate with each $\del$-complete set a welldefined Floer homology before we will come up with a solution for $\del$-incomplete sets.

Let $E \in \overline{\mcE}$ and consider $\Z E = \bigoplus_{k \in \Z} \Z E_k$ and endow it with the boundary operator defined in Equation \eqref{defDel}. Since $E$ is $\del$-complete, we have $\del^{E}_{k-1} \circ \del^{E}_k =0$ for all $k \in \Z$ in
\begin{equation}
\label{localChainComplex}
\dots \stackrel{\del^{E}_{k+1}}{\longrightarrow} \Z E_k  \stackrel{\del^{E}_k}{\longrightarrow} \Z E_{k-1} \stackrel{\del^{E}_{k-1}}{\longrightarrow} \Z E_{k-2}  \stackrel{\del^{E}_{k-2}}{\longrightarrow} \dots
\end{equation}
so that $(\Z E, \del^{E})$ is a chain complex. This observation gives rise to

\begin{definition}
\label{newDefLocalFH}
 The {\em Floer chain groups} of $E\in \overline{\mcE}$ are given by
 $$
 C_k(E):= C_k(\phi, x, E):= \Z E_k \quad \mbox{for all } k \in \Z
 $$
 and the boundary operator $\del^E_k:C_k(E) \to C_{k-1}(E)$ is defined on the generators $p \in E_k$ by
 $$
 \del^E_k(p):=\sum_{q \in E_{k-1}} n(p,q) \ q
 $$
 and is extended to $C_k(E)$ by linearity. The homology $H(E):=H(\phi, x, E)$ of the chain complex $(C(E), \del^E)$ is given in each degree by
 $$
 H_k(E):=H_k(\phi, x, E):= \frac{\ker \bigl(\del^E_k :C_k(E) \to C_{k-1}(E) \bigr)}{\Img \bigl(\del^E_{k+1}:C_{k+1}(E) \to C_{k}(E) \bigr)}
 $$
 and is called {\em (local) homoclinic Floer homology of $E \in \overline{\mcE}$}.
\end{definition}

\begin{remark}
 Choosing a different orientation on $W^u$ or working instead with orientations on $W^s$ may flip the sign of $n(p,q)$, i.e., instead of working with $\del$ we may end up working with $-\del$. Since the kernel and image of $\del$ and $-\del$ coincide, the resulting homology is independent of the choice of orientation. This also holds true when working with $W$-orientation reversing symplectomorphisms and $n_2(p,q)$. For more details see also Hohloch \cite[Section 3.3]{hohloch1}.
\end{remark}


\subsection{Pruning algorithm and Floer homology of $\del$-incomplete sets}
\label{pruningAlgo}

When a set $E \in \mcE$ is $\del$-incomplete then, according to \refgeomeDelComplete, there are cutting partners missing in $E$.


\subsubsection{First idea (unsuccessful): add missing cutting partners}
\label{sec:addCutPartners}

When cutting partners are missing, a natural idea is to add them to the original set to obtain a new set that is $\del$-complete. The main problem with this approach is that adding all missing cutting points may not be a finite process:

Let $E \in \mcE$ be $\del$-incomplete and consider its points with top Maslov degree $k_+:=k_+(E)$ so the first Maslov degree potentially containing cutting partners is $k_+-1$. Now assume that there are cutting partners missing in degree $k_+-1$. Now add them, thus making all hearts arisen from points of Maslov degrees $(k_+, k_+-1, k_+-2)$ $\del$-complete, and denote the new (finite) set by $A_1(E)$.

Now look in $A_1(E)$ for hearts with points of Maslov degrees $(k_+-1, k_+-2, k_+-3)$ with missing cutting partner in degree $k_+-2$. Add the necessary cutting partners and thus obtain a new (finite) set $A_2(E)$. The problem now is that these new points in degree $k_+-2$ may cause again trouble in degree $k_+-1$ by forming $\del$-incomplete hearts of degree $(k_+, k_+-1, k_+-2)$. But by creating the set $A_1(E)$ the $\del$-completeness in Maslov level $k_+-1$ was already supposed to be solved! Fixing this issue now may again create problems in Maslov level $k_+-2$ etc.

Similar problems arise analogously in Maslov levels lower than $k_+-1$ and $k_+-2$. Thus there is no guarantee that the process of adding missing cutting partners terminates within a finite number of steps, more precisely, we did not find a way to prove that it terminates.
So we have to come up with a different idea how to obtain a $\del$-complete set from a $\del$-incomplete one.


\subsubsection{Second idea (successful): delete points without cutting partners}
\label{sec:pruning}

Since adding the missing cutting points does not (yet?) work, see Section \ref{sec:addCutPartners}, let us now try to remove cutting points that do not have an associated cutting partner. Since there are only finitely many points in any given set $E \in \mcE \setminus \bar{\mcE}$, this procedure certainly will terminate within finitely many steps.

\vsp

\noindent
{\em Pruning algorithm:}
\begin{enumerate}[label=\arabic*)]
\item
Let $E \in \mcE$ with $k_+:=k_+(E)$ and $k_-:=k_-(E)$.
If $E \in \overline{\mcE}$ then $E$ is already $\del$-complete and nothing further needs to be done. If $E \in \mcE \setminus \overline{\mcE}$ then proceed to the following step.
 \item
 Since the boundary operator of a chain complex (and of a homology) always lowers the degree by one, we start from the top degree level downwards.
Now go systematically through all triples $(p, q, r) \in E_{k_+} \times E_{k_+ -1} \times E_{k_+-2}$ check if they give rise to a heart and, if yes, check wether the associated cutting partner $q'$ exists, i.e., if there is a triple $(p, q', r) \in E_{k_+} \times E_{k_+ -1} \times E_{k_+-2}$ that give rise to the same heart. If the triple gives rise to a heart, but there is no cutting partner $q'$ then delete $q$. Otherwise proceed until all triples have been checked once. Call the resulting set $P_1(E)$. Note that we have $0 \leq \abs{P_1(E)} \leq \abs{E} < \infty$.
\item
Repeat the previous step for the set $P_1(E)$. If there were no additional points deleted then call the unchanged set $P(E,{k_+})$ and continue to the next step. Otherwise call the set $P_2(E)$ and repeat this step (calling the new sets $P_3(E)$ etc.) until there are no points deleted any more. Note that we have $0 \leq \abs{P(E, {k_+})} \leq \abs{E} < \infty$ and in particular $0 \leq \abs{(P(E, {k_+}))_{k_+-1}} \leq \abs{E_{k_+-1}} \leq \abs{E}< \infty$. Thus there can only be finitely many repetitions.
\item
Now repeat the two previous steps for all triples in $ E_{k_+-k} \times E_{k_+ -(k+1)} \times E_{k_+-(k+2) }$ starting with $k=1$, then considering $k=2$ etc.\ until $k= k_+ -k_- -2$ is reached, i.e., until all triples including those in $ E_{k_-+2} \times E_{k_- +1} \times E_{k_- }$ are processed.
Since $0 \leq \abs{ P(E, k_-+2)} \leq  \abs{ P(E, k_+-k)} \leq \abs{ P(E, k_+)} \leq \abs{E}< \infty$ for all $k_+ \geq k \geq k_+ -k_- -2$ there can only be finitely many repetitions.
\item
Write $\overline{E}:= P(E, k_-+2)$ to shorten notation.
\end{enumerate}

This yields

\begin{theorem}
Let $E \in \mcE$. Then the set $\overline{E} $ obtained by the pruning algorithm is finite, $\del$-complete, and unique in the sense that checking, per level, the triples in a different order does not lead to a set $\overline{E}' $ different from $\overline{E} $.
\end{theorem}

\begin{proof}
{\em Finiteness:}
The set $\overline{E}:= P(E, k_-+2)$ resulting from the pruning algorithm is finite since
$$0 \leq \abs{\overline{E}} =\abs{ P(E, k_-+2)} \leq \abs{E} < \infty.$$

\noindent
{\em $\del$-completeness:}
After completing the pruning algorithm at level $k_+-k$, any heart formed by a triple $(p,q,r) \in E_{k_+-k} \times E_{k_+ -(k+1)} \times E_{k_+-(k+2) }$ gives rise to another triple $(p,q',r)$ where $q'$ is the cutting partner of $q$. Thus \refgeomeDelComplete\ is valid.

Moreover, when proceeding to level $k_+-(k+1)$ and performing the pruning algorithm, then maybe certain $r \in E_{k_+-(k+2)}$ will be deleted. But this does not destroy the property $\del_{k_+-k} \circ \del_{k_+-(k+1)}=0$ as we will see: consider a a heart containing $(p,q,r)$ and $(p,q',r)$, i.e., $q$ and $q'$ are cutting partners. Restricted to this heart, the boundary operator vanishes since
$$\del_{k_+-k} \circ \del_{k_+-(k+1)}(p) = \del_{k_+-k} (q-q') = \del_{k_+-k} (q) -  \del_{k_+-k} (q')=r-r =0.$$
After deleting the point $r$, we still find
$$\del_{k_+-k} \circ \del_{k_+-(k+1)}(p) = \del_{k_+-k} (q-q') =  \del_{k_+-k} (q) -  \del_{k_+-k} (q')= 0-0 =0.$$

{\em Uniqueness:} According to \refcut, cutting partners are unique. We now show that it does not matter when a point without cutting partner gets deleted: Assume that there is a heart formed by triples $(p,q,r)$ and $(p,q',r)$ containing the cutting partners $q$ and $q'$ and assume further that $q$ also appears in the heart formed by a triple $(\pti, q, \rti)$, but does not have a cutting partner $\qti$ forming the heart $(\pti, \qti, \rti)$.

If the algorithm happens to check first $(p,q',r)$ for potential lack of cutting partners, then, due to the existence of $(p,q,r)$, the point $q'$ is not deleted. But when $(\pti, q, \rti)$ is checked then the lack of the cutting partner $\qti$ causes the deletion of $q$. Then, when going though the next round of the pruning algorithm, $(p,q',r)$ now lacks its cutting partner $(p,q,r)$ so that $q'$ gets deleted, i.e., now $q$ and $q'$ are deleted.

If the algorithm happens to check $(p,q, r)$ first, then due to the existence of $(p,q',r)$, the point $q$ is not deleted. But when $(\pti, q, \rti)$ is checked the lack of the cutting partner $\qti$ causes the deletion of $q$. Then, in the next round of the pruning algorithm, $(p,q',r)$ lacks its cutting partner $(p,q,r)$ so that $q'$ gets deleted, i.e., now $q$ and $q'$ are deleted.

If the pruning algorithm happens to check $(\pti, q, \rti)$ first then $q$ gets deleted since $(\pti, q, \rti)$ has no cutting partner. Thus $(p,q',r)$ now lacks its cutting partner $(p,q,r)$ so that $q'$ gets deleted, i.e., now $q$ and $q'$ are deleted.
\end{proof}

Moreover

\begin{remark}
 Since the pruning algorithm deletes precisely those points that cause a set to be $\del$-incomplete, the resulting set is the largest $\del$-complete set contained in the original set.
\end{remark}

Since the pruning algorithm assigns to a $\del$-incomplete set $E$ a unique $\del$-complete set $\overline{E}$, we can now define Floer homology also for $\del$-incomplete sets:

\begin{definition}
\label{defLocalFH}
The {\em Floer chain groups} of $E\in \mcE$ are given by
 $$
 C_k(E):= C_k(\phi, x, E):= C_k(\phi, x, \overline{E})= \Z {\overline{E}}_k \quad \mbox{for all } k \in \Z
 $$
 and the boundary operator $\del^E_k: C_k(E) \to C_{k-1}(E)$ is defined on the generators $p \in {\overline{E}}_k$ by
 $$
 \del^E_k(p):=\del^{\overline{E}}_k (p)= \sum_{q \in \overline{E}_{k-1}} n(p,q) \ q
 $$
 and is extended to $C_k(E)$ by linearity. The homology $H(E):=H(\phi, x, E)$ of the chain complex $(C(E), \del^E)$ is given in each degree by
 $$
 H_k(E):=H_k(\phi, x, E):= \frac{\ker \bigl(\del^E_k :C_k(E) \to C_{k-1}(E) \bigr)}{\Img \bigl(\del^E_{k+1}:C_{k+1}(E) \to C_{k}(E) \bigr)}
 $$
and is called {\em (local) homoclinic Floer homology of $E \in \mcE$}. Note that this definition agrees for $\del$-complete sets $E$ with the original \refnewDefLocalFH.
\end{definition}

Since $H(E)$ only takes the homoclinic points in $E$ resp.\ $\overline{E}$ into account, these chain groups and homology groups measure only `local' properties of $\mcH$, namely `behaviour near' $E$ resp.\ $\overline{E}$ which are `very small' subsets of $\mcH$.


\section{Background from homological algebra}
\label{sec:backgroundHomAlg}

In what follows, we summerize some definitions and facts from homological algebra that we will need later on. We refer to Rotman's \cite{rotman} and Weibel's \cite{weibel} text books for details. Whereas the first book mainly works with directed sets and systems, the latter one formulates it more generally for filtered categories.


\subsection{Left $R$-modules}

Since we are later on mainly interested in freely generated groups and their quotients it suffices to recall notions from homological algebra within the following setting.

\begin{definition}
Let $R$ be a ring with multiplicative identity $1_R$.
 A {\em left $R$-module $T$} consists of an abelian group $(T , +)$ endowed with an operation $R \x T \to T$, $(r, p) \mapsto r.p$ satisfying for all $r,s \in R$ and all $p,q \in T$
 \begin{enumerate}[label=\arabic*)]
  \item
  $r.(p+q) = r.p + r.q$,
  \item
  $(r+s).p = r.p + s.p$,
  \item
  $(rs).p = r.(s.p)$,
  \item
  $1_R.p = p$.
 \end{enumerate}
\end{definition}

\noindent
Thus, intuitively, a module is a `vector space over a ring'.

\begin{remark}
 Let $(T, +)$ be an abelian group and $\Z$ the ring of integers. By means of
 $$
 \Z \x T \to T, \quad (n, p) \mapsto n.p:=
 \left\{
 \begin{aligned}
  &  \underbrace{p + \dots + p}_{n}, \quad \mbox{for } n \in \Z^{\geq 0}, \\
  & -((-n).p), \quad \mbox{for } n \in \Z^{<0}
 \end{aligned}
 \right.
$$
we turn $T$ into a $\Z$-modul, i.e., every abelian group can be seen as a $\Z$-module.
\end{remark}




\subsection{Direct limit}

Let us fix some notions and notations.

\begin{definition}
A set $I$ with a binary relation $\preceq$ is said to be a {\em partially ordered set}, briefly called a {\em poset}, if
\begin{enumerate}[label=\arabic*)]
 \item
 \ $\preceq$ is reflexive, i.e.,
 $\forall\ i \in I: i \preceq i$.
 \item
 \ $\preceq$ is antisymmetric, i.e., if $i \preceq j$ and $j\preceq i$ then $i=j$.
 \item
 \ $\preceq$ is transitive, i.e.,
 $\forall\ i,j,k \in I$ with $i \preceq j$ and $j \preceq k$ follows $i \preceq k$.
\end{enumerate}
A partially ordered set $(I, \preceq)$ is said to be {\em directed} if for all $i, j \in I$, there exists $k=k(i,j) \in I$ such that $i \preceq k$ and $j \preceq k$, i.e., each pair of elements has a `common upper bound'.
\end{definition}

\begin{definition}
 Let $(I, \preceq)$ be a partially ordered set and $\msC$ a category. A {\em direct system in $\msC$} is an indexed family
 $$
 \{ C^i, \ga^{ij} \} := \left( (C^i)_{i \in I}, (\ga^{ij})_{\stackrel{i, j \in I}{i \preceq j}} \right)
 $$
 where $C^i \in \Obj(\msC)$ for all $i \in I$ and $\ga^{ij} \in \Morph(C^i, C^j)$ for all $i, j \in I$ with $i \preceq j$ satisfying in addition
 \begin{enumerate}[label=\arabic*)]
  \item
    \ $\ga^{ii}= \Id_{C^i} $ for all $i \in I$.
  \item
  \ $\ga^{jk} \circ \ga^{ij} = \ga^{ik}$ for all $i, j, k \in I$ with $i \preceq j \preceq k$.
  \end{enumerate}
\end{definition}

\noindent
Given such a direct system, we are interested in the following.

\begin{definition}
 Let $(I, \preceq)$ be a partially ordered set, $\msC$ a category, and $\{ C^i, \ga^{ij} \}$ a direct system in $\msC$. The associated {\em direct limit} (if it exists) is given by an object $\varinjlim C^i \in \Obj(\msC) $ together with a family of morphisms $\ga^k \in \Morph(C^k, \varinjlim C^i)$ for all $k \in I$ such that
 \begin{enumerate}[label=\arabic*)]
  \item
  $\ga^j \circ \ga^{ij} = \ga^i$ for all $i, j \in I$ with $i \preceq j$.
  \item
  Given $\Cti \in \Obj(\msC)$ with $\gati^i \in \Morph(C^i, \Cti)$ for all $i \in I$ and satisfying $\gati^j \circ \ga^{ij} = \gati^i$ for all $i, j \in I$ with $i \preceq j$ then there exists a unique $\theta \in \Morph( \varinjlim C^i, \Cti)$ such that $\theta \circ \ga^i = \gati^i$ for all $i \in I$.
 \end{enumerate}
The direct limit is often also called {\em inductive limit} or {\em colimit}.
\end{definition}

If the direct limit exists it is therefore unique up to (unique) isomorphism. Rotman shows in \cite[Proposition 5.23, Lemma 5.30, Corollary 5.31]{rotman} the following.

\begin{proposition}
\label{existenceDirectLim}
Let $R$ be a ring, $(I, \preceq)$ a partially ordered set and $\{C^i, \ga^{ij}\}$ a direct system of left $R$-modules. For $i \in I$, denote by $\lam^i : C^i \to \bigoplus _{i \in I} C^i$ the insertion morphisms.
\begin{enumerate}[label=\arabic*)]
 \item
 Then the direct limit $\varinjlim C^i $ exists and is given by
$$
\varinjlim C^i = \left. \left( \bigoplus _{i \in I} C^i \right) \right\slash S
$$
where $S$ is the submodule of $ \bigoplus _{i \in I} C^i$ generated by all elements of the form $\lam^j(\ga^{ij}(c^i)) - \lam^i(c^i)$ with $i, j \in I$ and $i \preceq j$.
 \item
 If $(I, \preceq)$ is in addition directed, then the following holds true:
 \begin{itemize}
  \item
  Each element of $\varinjlim C^i$ has a representative of the form $\lam^i ( c^i) + S$ (instead of $\sum_{i}\lam^i (c^i) + S$) where $c^i \in C^i$.
  \item
  For all $i \in I$ and all $c^i \in C^i$, we have $\lam^i( c^i) + S =0$ if and only if there exists $k=k(i) \in I$ with $i \preceq k$ such that $\ga^{ik} (c^i) =0$.
  \item
  Given $c^i \in C^i$ and $c^j \in C^j$ for $i, j \in I$, we define $c^i \sim c^j$ if and only if there exists $k=k(i,j) \in I$ with $i \preceq k$ and $j \preceq k$ such that $\ga^{ik}(c^i) = \ga^{jk}(c^j)$. This is an equivalence relation on the disjoint union $\bigsqcup_{i \in I} C^i$ giving rise to the isomorphism
  $$
  \varinjlim C^i \quad \simeq \quad \left. \left( \bigsqcup _{i \in I} C^i \right)\right\slash \sim
  $$
  Thus elements of $\varinjlim C^i$ can be seen as equivalence classes $\llbracket c^i \rrbracket$ where $c^i \in C^i$ with addition defined by
  $$\llbracket c^i \rrbracket + \llbracket c^j \rrbracket = \llbracket \ga^{ik}(c^i) + \ga^{jk}(c^j) \rrbracket.$$
 \end{itemize}
\end{enumerate}
\end{proposition}

\begin{definition}
Let $(I, \preceq)$ be a partially ordered set and $A:=\{A^i, \al^{ij}\}$ and $B:=\{B^i, \be^{ij}\}$ direct systems over $(I, \preceq)$. A {\em morphism of direct systems} is a map $g: A \to B$ consisting of maps $g^i: A^i \to B^i$ for all $i \in I$ such that the following diagram commutes for all $i, j \in I$ with $i \preceq j$:
$$
\begin{array}{ccc}
A^i & \stackrel{g^i}{\longrightarrow} & B^i \\
{\scriptstyle \al^{ij}} \downarrow & & \downarrow {\scriptstyle  \be^{ij}} \\
A^j & \stackrel{g^j} {\longrightarrow}& B^j
\end{array}
$$
\end{definition}

\begin{remark}
Let $(I, \preceq)$ be a partially ordered set and $A:=\{A^i, \al^{ij}\}$ and $B:=\{B^i, \be^{ij}\}$ direct systems over $(I, \preceq)$. Recall from \refexistenceDirectLim\ the notations $\varinjlim A^i = \left(\oplus_{i \in I} A^i \right) \slash S^A$ and $\varinjlim B^i = \left(\oplus_{i \in I} B^i \right) \slash S^B$ and the insertion morphisms $\lam^i_A: A^i \to \oplus_{i \in I} A^i$ and $\lam^i_B: B^i \to \oplus_{i \in I} B^i$.
Then a morphism of direct systems $g: A \to B$ gives rise to a morphism
$$
\underscriptrightarrow{g}: \varinjlim A^i \to \varinjlim B^i, \qquad \underscriptrightarrow{g} \left(\sum \lam^i_A (a^i) + S^A \right) :=  \sum \lam^i_B (g^i(a^i)) + S^B.
$$
\end{remark}

Moreover, we will need the following important property of the direct limit within the category of left $R$-modules when taken over directed sets:

\begin{proposition}[\cite{rotman}, Proposition 5.33]
\label{shortDirectLim}
 Let $(I, \preceq)$ be a directed set and $A:=\{A^i, \al^{ij}\}$ and $B:=\{B^i, \be^{ij}\}$ and $C:=\{C^i, \ga^{ij}\}$ direct systems of left $R$-modules over $(I, \preceq)$. Let $g: A \to B$ and $h: B \to C$ be morphisms of direct systems forming a short exact sequence of direct systems
 $$
 0 \to A \stackrel{g}{\to} B \stackrel{h}{\to} C \to 0,
 $$
 i.e., for all $i \in I$,
 $$
 0 \to A^i \stackrel{g^i}{\to} B^i \stackrel{h^i}{\to} C^i \to 0
 $$
is a short exact sequence. Then there is the short exact sequence
 $$
 0 \longrightarrow \varinjlim A^i \stackrel{\underscriptrightarrow{g}}{\longrightarrow} \varinjlim B^i \stackrel{\underscriptrightarrow{h}}{\longrightarrow} \varinjlim C^i \longrightarrow 0
 $$
 \end{proposition}

 We note

 \begin{lemma}
 \label{quotientDirSys}
 Let $(I, \preceq)$ be a directed set. Let $A:=\{A^i, \al^{ij}\}$ and $B:=\{B^i, \be^{ij}\}$ be direct systems of left $R$-modules over $(I, \preceq)$ such that there is a morphism of direct systems $g: A \to B$ with $ g^i: A^i \to B^i $ injective for all $i \in I$.
 Then
 \begin{enumerate}[label=\arabic*)]
  \item
  $ 0 \to A^i \stackrel{g^i}{\to} B^i \stackrel{h^i}{\to} B^i \slash g^i(A^i) \to 0$ is a short exact sequence for all $i \in I$ where $h^i: B^i \to B^i \slash g^i(A^i)$ is given by $h^i(b^i):= b^i + g^i(A^i)$.
  \item
  $B\slash g(A) := \{ B^i \slash g^i(A^i), \ga^{ij}\}$ with transition maps $\ga^{ij}: B^i \slash g^i(A^i) \to B^j \slash g^j(A^j)$ defined by $\ga^{ij}(b^i + g^i(A^i)):= \be^{ij}(b^i) + g^j(A^j)$ is a direct system over $(I, \preceq)$.
  \end{enumerate}
This gives rise to a morphism of direct systems $h: B \to B\slash g(A)$ and a short exact sequence of direct systems
$$
0 \to A \stackrel{g}{\to} B \stackrel{h}{\to} B\slash g(A) \to 0
$$
\end{lemma}

 \begin{proof}
{\bf 1)} is clear.

\noindent
{\bf 2)} We have $\ga^{ii}=\Id_{B^i \slash g^i(A^i)}$ for all $i \in I$ and we compute for all $i, j, k \in I$ with $i \preceq j \preceq k$:
\begin{align*}
\ga^{jk}(\ga^{ij}(b^i + g^i(A^i))) & = \ga^{jk}(\be^{ij}(b^i) + g^j(A^j)) = \be^{jk}(\be^{ij}(b^i)) + g^k(A^k) = \be^{ik}(b^i) + g^k(A^k) \\
& = \ga^{ik}(b^i + g^i(A^i)) .
\end{align*}
Moreover, we find for all $i, j \in I$ with $i \preceq j$
\begin{align*}
 h^j(\be^{ij}(b^i)) & = \be^{ij}(b^i) + g^j(A^j) = \ga^{ij}(b^i + g^i(A^i)) = \ga^{ij}(h^i(b^i)) .
\end{align*}
\end{proof}

This implies in particular

\begin{corollary}
\label{dirLimQuotient}
 Let $(I, \preceq)$ be a directed set. Let $A:=\{A^i, \al^{ij}\}$ and $B:=\{B^i, \be^{ij}\}$ direct systems of left $R$-modules over $(I, \preceq)$ such that there is a morphism of direct systems $g: A \to B$ such that $ g^i: A^i \to B^i $ is injective for all $i \in I$.
 Then, up to unique isomorphism, we have
 $$
 \varinjlim \left( B^i \slash g^i(A^i) \right) = \varinjlim B^i \slash \varinjlim g^i(A^i)
 $$
\end{corollary}

\begin{proof}
 Completing the sequence of direct systems $0 \to A \stackrel{g}{\to} B$ to the short exact sequence of direct systems $0 \to A \stackrel{g}{\to} B \stackrel{h}{\to} B\slash g(A) \to 0 $ as in \refquotientDirSys, \refshortDirectLim\ implies the existence of the short exact sequence
 $$
 0 \to \varinjlim A^i \stackrel{\underscriptrightarrow{g}}{\longrightarrow} \varinjlim B^i \stackrel{\underscriptrightarrow{h}}{\longrightarrow} \varinjlim (B^i\slash g^i(A^i)) \to 0
 $$
 On the other hand, completing the sequence $0 \to \varinjlim A^i \stackrel{\underscriptrightarrow{g}}{\longrightarrow} \varinjlim B^i $ to a short exact sequence yields
 $$
 0 \to \varinjlim A^i \stackrel{\underscriptrightarrow{g}}{\longrightarrow} \varinjlim B^i \longrightarrow \varinjlim B^i \slash \varinjlim g^i(A^i) \to 0
 $$
 Since the direct limit is unique up to unique isomorphism this yields the claim.
\end{proof}


\section{Homoclinic Floer homology via direct limit}
\label{sec:directLimFH}

In Section \ref{sec:localFH}, we defined local Floer homology generated by finite sets of homoclinic points. In this section, we will use this `local' notion in order to define a `more global' version by passing to the direct limit.
This significantly generalizes {\em (semi)primary homoclinic Floer homology} (cf.\ Hohloch \cite{hohloch1, hohloch2, hohloch3}) that is generated by the finite set of {\em (semi)primary} homoclinic orbits.


\subsection{Direct systems of $\Z$-modules generated by homoclinic points}

\begin{lemma}
Endow the set $\mcE$ with the inclusion $\subseteq$ as binary relation.
Then $(\mcE, \subseteq)$ is a directed set.
\end{lemma}

\begin{proof}
 $(\mcE, \subseteq)$ is clearly a partially ordered set. Given $E_1, E_2 \in \mcE$, we note that $E_1, E_2 \subseteq (E_1 \cup E_2) \in \mcE$ so that $(\mcE, \subseteq)$ is in fact a directed set.
\end{proof}

For $D, E \in (\mcE, \subseteq)$ with $D \subseteq E$, denote the inclusion induced on the $\Z$-modules by
$$\mcI^{ D E}: \Z D \to \Z E.$$

\begin{lemma}
The family
\begin{align*}
 \Z \mcE &:= \bigl( \{ \Z E \mid E \in \mcE\} , \{\mcI^{ D  E} \mid D, E \in \mcE,\ D \subseteq E \} \bigr)
\end{align*}
with index set $\mcE$ is a direct system of $\Z$-modules in the category of $\Z$-moduls. In particular, it is a direct systems of abelian groups.
%
\end{lemma}

\begin{proof}
 We have $\mcI^{E E}= \Id_{\Z E}$ for all $E \in \mcE$ and, for sets $D, E, F  \in \mcE$ with $D \subseteq E \subseteq F$ we find moreover $\mcI^{EF} \circ \mcI^{DE} = \mcI^{DF}$.
\end{proof}

As discussed in Section \ref{sec:problemsDelComplete}, not all sets in $\mcE$ admit a welldefined boundary operator, only $\del$-complete sets do. So it is of interest if $(\overline{\mcE}, \subseteq) \subset ({\mcE}, \subseteq)$ is also partially ordered and directed.

\begin{lemma}
 $(\overline{\mcE}, \subseteq)$
 is a partially ordered set.
\end{lemma}

\begin{proof}
The claim is true since $(\mcE, \subseteq)$ induces a partial order on $\overline{\mcE} \subset \mcE$.
\end{proof}

We note

\begin{example}
 There exist subsets of $(\overline{\mcE}, \subseteq)$ that are directed.
\end{example}

\begin{proof}
 Using the notation from Figure \ref{sameCutPart}, a very easy (finite) example is given by $\bigl\{ \{p\}, \{ r\},  \{q_a, q_b\},\{p, q_a, q_b, r\} \bigr\} \subset \overline{\mcE}$.
\end{proof}

We strongly suspect that $(\overline{\mcE}, \subseteq)$ is also directed, i.e., that for every $D, E \in \overline{\mcE}$ there is $F \in \overline{\mcE}$ with $D$, $E \subseteq F$, but we have been so far unable to prove it. Note that if the procedure in Section \ref{sec:addCutPartners} (which adds cutting partners in order to achieve $\del$-completeness) could be made to work then taking $F$ to be the $\del$-completion of $D \cup E$ would do the trick. But there may be other ways to $\del$-compete a set.

\newpage

\begin{proposition}
\label{propModSys}
\mbox{ } 
 \begin{enumerate}[label=\arabic*)]
 \item
 The family
$
\Z \overline {\mcE}:= \bigl( \{ \Z E \mid E \in \overline{\mcE} \}, \{\mcI^{ D E} \mid D, E \in \overline{\mcE}, D \subseteq E \}  \bigr)
$
with index set $\overline{\mcE}$ is a direct system of $\Z$-moduls in the category of $\Z$-moduls. Therefore the direct limit of $\Z \overline {\mcE} $ with index set $\overline{\mcE}$ exists in the category of $\Z$-moduls. Note that, in particular, $\Z \overline {\mcE}$ is a direct system of abelian groups.
\item
Let $\tilde{\mcE} \subset \overline{\mcE}$ be directed. Then the family
$$
\Z \tilde{\mcE}:= \bigl( \{ \Z E \mid E \in \tilde{\mcE} \}, \{\mcI^{ D E} \mid D, E \in \tilde{\mcE}, D \subseteq E \}  \bigr)
$$
with index set $\tilde{\mcE}$ is a direct system of $\Z$-moduls in the category of $\Z$-moduls and, in particular, a direct system of abelian groups. Moreover, since $\tilde{\mcE}$ is directed, there is a more explicit expression for the direct limit of $\Z \tilde{\mcE}$ in the category of $\Z$-moduls and, in addition, compatibility with taking short exact sequences and quotients of direct limits.
\end{enumerate}
\end{proposition}

\begin{proof}
Since $(\overline{\mcE}, \subseteq)$ is a partially ordered set, the family $\Z \overline{\mcE}$ indexed by $\overline{\mcE}$ is a directed system. The existence of the direct limit follows from \refexistenceDirectLim.
The corresponding claim concerning $\Z \tilde{\mcE}$ follows analogously.

Since the index set $\tilde{\mcE}$ is directed, \refexistenceDirectLim\ provides not only the existence of a direct limit, but also a more explicit form of it. \refshortDirectLim\ now provides the ability to pass to the direct limit of short exact sequences and \refdirLimQuotient\ the compatibility of taking a quotient and passing to the direct limit when working with short exact sequences.
\end{proof}


\subsection{Passing from $\Z$-modules to chain complexes}

Unfortunately, without additional assumptions, \refpropModSys\ only holds for (graded) $\Z$-modules, but does not extend to chain complexes in general, more precisely

\begin{lemma}
 \label{noChainMap}
 \mbox{ } 
 \begin{enumerate}[label=\arabic*)]
  \item 
  There are $D, E \in \overline{\mcE}$ with $D \subset E$ where $\mcI^{D E}: \Z D \to \Z E$ does not induce a chain map $\mcI^{D E}: (\Z D, \del^{D}) \to (\Z E, \del^{E})$, i.e, where we have
  $$\mcI^{D E} \circ \del^{D} \ \neq \ \del^{E} \circ  \mcI^{D E }.$$
  \item
   There exist $D, E \in \overline{\mcE}$ with $D \subset E$ such that there is no nontrivial linear map from $\Z D$ to $\Z E$ giving rise to a chain map from $(\Z D, \del^{D})$ to $(\Z E, \del^{E})$.
 \end{enumerate}
\end{lemma}

Examples for \refnoChainMap\ are sketched in Figure \ref{figNoChainMap}.

\begin{figure}[h]
\begin{center}

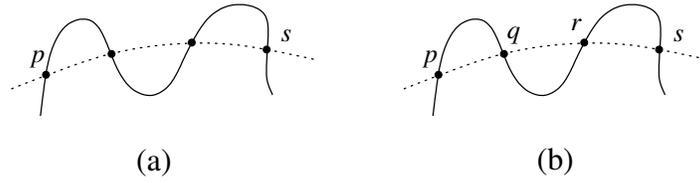

\caption{The stable manifold is drawn with a continuous line and the unstable manifold with a dotted line and the latter is oriented `from left to right'. Subfigure (a) displays $D:=\{p, s\}$ with $\mu(p) = \mu(q)+1$. The set $D$ is $\del$-complete due to $\del^{D} p = 0$ and $\del^{D} s=0$. Subfigure (b) shows the set $E:=\{p, q, r, s\}$ with $\mu(p) = \mu(r)=\mu(q) + 1 =\mu(s)+ 1$. This set is $\del$-complete due to $\del^{E} p = q$ and $\del^{E} r = s -q$ and $\del^{E} q = 0$ and $\del^{E} s =0$.}
\label{figNoChainMap}

\end{center}
\end{figure}

\begin{proof}[Proof of \refnoChainMap:]
 Examples for both statements are sketched in Figure \ref{figNoChainMap}: Consider the sets $D:=\{p,s\}$ in Subfigure \ref{figNoChainMap} (a) and $E:=\{p,q,r,s\}$ in Subfigure \ref{figNoChainMap} (b) and consider their chain complexes $(\Z D, \del^{D})$ and $(\Z E, \del^{E})$. We note $\mu(p) = \mu(r) =\mu(q) +1 =\mu(s) +1 $ and abbreviate $\mu(p)=\mu(r)=:k$.
 
 {\bf 1)}
 Now have a look at the following (noncommutative!) diagram of chain complexes where the vertical arrows are induced by the inclusion $\mcI^{D E}: \Z D \to\Z E$, i.e., are given by the family $(\mcI^{D E}_k)_{k \in \Z}$ with $ \mcI^{D E}_k  : \Z D_k \to \Z E_k$:
 $$
 \begin{array}{cclclclclcc}
  \dots & \stackrel{\del_{k+2}^{D}}{\longrightarrow} & 0 & \stackrel{\del_{k+1}^{D}}{\longrightarrow} & \Z\{p\} & \stackrel{\del_{k}^{D}}{\longrightarrow} & \Z \{s\} & \stackrel{\del_{k-1}^{D}}{\longrightarrow} & 0 &   \stackrel{\del_{k-2}^{D}}{\longrightarrow} & \dots  \\
  &&&&&&&&&& \\
  && \downarrow \mcI^{D E}_{k+1} & &  \downarrow \mcI^{D E}_{k}&& \downarrow \mcI^{D E}_{k-1}&& \downarrow \mcI^{D E}_{k-2}&&  \\
  &&&&&&&&&& \\
  \dots & \stackrel{\del_{k+2}^{E}}{\longrightarrow} & 0 & \stackrel{\del_{k+1}^{E}}{\longrightarrow} & \Z\{p, r\} & \stackrel{\del_{k}^{E}}{\longrightarrow} & \Z \{q, s\} & \stackrel{\del_{k-1}^{E}}{\longrightarrow} & 0 &   \stackrel{\del_{k-2}^{E}}{\longrightarrow} & \dots
 \end{array}
 $$
 It is noncommutative since for example
 $$
 \mcI^{D E}_{k-1}(\del_{k}^{D} (p)) = \mcI^{D E}(0)=0 \neq q = \del_{k}^{E} (p) = \del_{k}^{E}( \mcI^{D E}_k(p))
 $$
 i.e., $\mcI^{D E}$ is no chain map.
 
 {\bf 2)}
 Consider the grading preserving linear map $g=(g_k)_{k \in \Z}: \Z D \to \Z E$ given on the generators by $g_k(p) := \al p+ \beta r$ with $\al, \be \in \R$ and $(\al, \be) \neq (0,0)$ and $g_{k-1}(s) \in \Z E$ arbitrary and $g_\ell =0$ for all $\ell \notin \{k, k+1\}$. Again, the induced diagram of chain complexes will not commute since
 $$
 g_{k-1}(\del_{k}^{D} (p))=g_{k-1}(0)=0 \neq \al q + \be (s-q) =  \del_{k}^{E} (p+r) =\del_{k}^{E} (g_k(p)).
 $$
 Thus there is in fact no grading preserving nontrivial linear map from $\Z D$ to $\Z E$ that gives rise to a chain map from $(\Z D, \del^D)$ to $(\Z E, \del^E)$. 
\end{proof}

Thus, if we want to work with chain complexes and not just $\Z$-modules, we have to use direct systems where the transition maps are chain maps.

\begin{definition}
 Let $D, E  \in \overline{\mcE}$ with $D \subseteq E$. The pair $(D, E)$ is said to be {\em chain compatible} if $\mcI^{D E}_{k-1} \circ \del_k^{D} = \del_k^{E} \circ \mcI_k^{D E}$ for all $k \in \Z$, i.e., if the inclusion $\mcI^{D E}$ is a chain map.
\end{definition}

The set of chain compatible pairs is not empty. A very simple example is

\begin{example}
Let $D, E \in \overline{\mcE}$ with $D \subseteq E$ and $\del^{D}=0 = \del^{E}$. Then the pair $(D, E)$ is chain compatible. An explicit example for such a situation is any pair $(D, E)$ with $D \subseteq E$ and $D =D_k$ and $E = E_k$ for a given $k \in \Z$.
\end{example}

To be able to formulate a criterion for chain compatibility let us fix the following notation.
Consider $D, E \in \overline{\mcE}$ with $D \subseteq E$. Then we get for $p \in E$
 \begin{align*}
 \del^E p & = \sum_{\stackrel{q \in E}{\mu(q)=\mu(p)-1}} n(p,q) \ q  =  \sum_{\stackrel{q \in D}{\mu(q)=\mu(p)-1}} n(p,q) \ q + \sum_{\stackrel{q \in E\setminus D}{\mu(q)=\mu(p)-1}} n(p,q) \ q
  =: \del^{E_{D}} (p)+ \del^{E_{E \setminus D}}(p),
\end{align*}
i.e., the boundary operator splits into $\del^E = \del^{E_{D}} + \del^{E_{E \setminus D}}$.

\begin{lemma}[Criterion]
\label{criterionChainCompat}
 Let $D, E \in \overline{\mcE}$ with $D \subseteq E$. Then 
 $$ \mbox{The pair } (D, E) \mbox{ is chain compatible }
 \ \ \Leftrightarrow \ \
 \del^{E_{E \setminus D}}\circ \mcI^{D E} =  0, \mbox{ i.e., }\del^{E_{E \setminus D}} \mbox{ vanishes on } D \subseteq E.
 $$
\end{lemma}

\begin{proof}
We have to check when $\mcI^{D E}$ is a chain map, i.e., when we have
$\mcI^{D E}_{k-1} \circ \del_k^{D} = \del_k^{E} \circ \mcI_k^{D E}$ for all $k \in \Z$.
We compute for $p \in D \subseteq E$
 \begin{align*}
 \mcI^{D E}(\del^{D}(p)) & = \mcI^{D E} \left(\sum_{\stackrel{q \in D}{\mu(q)=\mu(p)-1}} n(p,q) \ q \right) =\sum_{\stackrel{q \in D}{\mu(q)=\mu(p)-1}} n(p,q) \ q 
 \end{align*}
 and compare it with
 \begin{align*}
 \del^{E}(\mcI^{D E}(p)) 
 & =  \del^{E_D} (\mcI^{D E}(p)) + \del^{E_{E \setminus D}} (\mcI^{D E}(p))
 = \sum_{\stackrel{q \in D}{\mu(q)=\mu(p)-1}} n(p,q) \ q + \sum_{\stackrel{q \in E \setminus D}{\mu(q)=\mu(p)-1}} n(p,q) \ q 
\end{align*}
which yields the claim.
\end{proof}

Now we extend the notion of chain compatibility to direct systems.

\begin{definition}
 $\tilde{\mcE} \subset (\overline{\mcE}, \subseteq)$ is said to be {\em chain compatible} if, for all $D, E \in \tilde{\mcE}$ with $D \subseteq E$, the pair $(D, E)$ is chain compatible.
\end{definition}

The set of chain compatible direct systems is not empty. A very simple example is

\begin{example}
Let ${\mcE'} \subset (\overline{\mcE}, \subseteq)$ be a subset satisfying $\del^{E}=0$ for all $E \in {\mcE'}$. Let $D, E \in {\mcE'} $ with $D \subseteq E$. Since $0= \del^E = \del^{E_{D}} + \del^{E_{E \setminus D}}$ we get $\del^{E_{E \setminus D}} =0$ so that ${\mcE'}$ is chain compatible.
Note that having $\del^E=0$ for all $E \in {\mcE'}$ is, for example, the case if $E=E_k$ for all $E \in {\mcE'}$ and a given $k \in \Z$.
\end{example}

\subsection{Passing from chain complexes to direct limits of homoclinic Floer homology}

The first aim of this subsection is to obtain direct systems of (local) homoclinic Floer homologies of which later the direct limit can be taken.

\begin{proposition}
\label{PropHomSys}
Let $D, E \in \overline{\mcE}$ with $D \subseteq E$ and $(D, E)$ chain compatible.
Then the chain map $\mcI^{DE}: \Z D \to \Z E$ gives rise to a map in homology $H(\mcI^{DE}): H(D) \to H(E)$.
Moreover, for all $D, E, F \in \overline{\mcE}$ with $D \subseteq E \subseteq F$ and $(D, E)$, $(E, F)$, $(D, F)$ chain compatible, we have
$$H(\mcI^{EF}) \circ H(\mcI^{DE}) = H(\mcI^{DF}).$$
\end{proposition}

\begin{proof}
Let $D, E \in \overline{\mcE}$ with $D \subseteq E$ and $(D, E)$ chain compatible. Due to the chain compatibility of $(D, E)$, the map $\mcI^{DE}: \Z D \to \Z E$ is in fact a chain map, i.e., equipping the $\Z$-modules with differentials $(\Z D, \del^D)=:(C(D), \del^D)$ and $(\Z E, \del^E)=:(C(E), \del^E)$ the map $\mcI^{DE}$ satisfies $\del^E_k \circ \mcI^{DE}_k =   \mcI^{DE}_{k-1} \circ \del^D_k$ for all $k \in \Z$ and is therefore also a chain complex map $\mcI^{D, E}= C(\mcI^{D, E}): (C(D), \del^{D}) \to (C(E), \del^{E})$. Moreover, being an inclusion, $C(\mcI^{D, E}): (C(D), \del^D) \ \hookrightarrow \ (C( E), \del^E)$ extends to a short exact sequence of chain complexes
$$
0 \to (C(D), \del^D) \ \stackrel{\ \ C(\mcI^{D, E})}{\hookrightarrow} \ (C( E), \del^E) \ \twoheadrightarrow \ ( C(E, D), \del^{(E,D)}) \ \to 0
$$
This gives rise to a long exact sequence of homology groups
 $$
  \cdots \to H_{k+1}(E, D) \to H_k(D) \stackrel{H_k(\mcI^{DE})}{\to} H_k(E) \to H_k(E, D) \to H_{k-1}(D) \stackrel{H_{k-1}(\mcI^{DE})}{\to} \cdots
 $$
with maps $H_k(\mcI^{DE}): H_k(D) \to H_k(E)$ for all $k \in \Z$, i.e., we obtain a map
$$H(\mcI^{DE}): H(D) \to H(E).$$
Now consider $D, E, F \in \overline{\mcE}$ with $D \subseteq E \subseteq F$ and $(D, E)$, $(E, F)$, $(D, F)$ being chain compatible.
Recall that passing from the category of short exact sequences of chain complexes in an abelian category by taking the homology to the category of long exact homology sequences is a functor. This means that a morphism between chain complexes is mapped, when passing to the homology sequence, to a morphism between long exact sequences (cf.\ Weibel \cite[Theorem 1.3.4]{weibel}).
Thus the following commutative diagram of chain complexes
 $$
 \begin{array}{ccccccccccc}
  0 & \to & (C( D), \del^D) &  \stackrel{C(\mcI^{DE})}{\hookrightarrow} & (C(E), \del^E) & \twoheadrightarrow & (C(E, D), \del^{E,D}) & \to & 0 \\
  &&&&&&&& \\
  \parallel & & \parallel & & \downarrow {C( \mcI^{EF})} & & \downarrow  && \parallel \\
  &&&&&&&& \\
  0 & \to & (C( D), \del^D) &  \stackrel{C(\mcI^{DF})}{\hookrightarrow} & (C( F), \del^F) & \twoheadrightarrow & (C(F, D), \del^{F,D}) & \to & 0
 \end{array}
$$
can be seen as morphism between two short exact sequences of chain complexes. It gives rise to a morphism between the long exact sequences, namely the following commuting diagram of long exact sequences
$$
\begin{array}{ccccccccccc}
 \cdots &  \longrightarrow &  H_{k+1}(D, C) & \longrightarrow & H_k(C) & \stackrel{H_k(\mcI^{CD})} {\longrightarrow} &  H_k(D) &  \longrightarrow &  H_k(D, C) &  \longrightarrow &  \cdots \\
 &&&&&&&&&& \\
 && \downarrow && \downarrow \Id && \downarrow H_k(\mcI^{DE}) && \downarrow && \\
 &&&&&&&&&& \\
 \cdots & \longrightarrow &  H_{k+1}(C, E) & \longrightarrow & H_k(C) & \stackrel{H_k(\mcI^{CE})}{\longrightarrow} & H_k(E) & \longrightarrow & H_k(E, C) & \longrightarrow & \cdots \\
 &&&&&&&&&& \\
\end{array}
$$
This shows $H(\mcI^{EF}) \circ H(\mcI^{DE}) = H(\mcI^{DF})$.
Moreover, we also have $\Id_{H_k(D)} = H_k(\mcI^{DD})$.
\end{proof}

Since the local homoclinic Floer homology assigned to a $\del$-complete, finite set $E$ only can observe a very small part of the information contained in the full set of homoclinic points $\mcH$ itself, it is a natural idea to make the set $E$ `bigger' in order to gather more of the information stored in $\mcH$. So opting for taking a direct limit over a suitable directed index set is a logical choice. It also would be compatible with the attempt of `adding cutting points to make $\del$-incomplete sets $\del$-complete' discussed in Section \ref{sec:problemsDelComplete}.

\begin{theorem}
\label{thDirectLim}
Let $\tilde{\mcE} \subset (\overline{\mcE}, \subseteq)$ be chain compatible. Then, for all $k \in \Z$,
 \begin{align*}
 H^{\ti{\mcE}}_k& := \bigl( \{ H_k(D) \mid D   \in \tilde{\mcE}\} , \{ H(\mcI^{DE}_k) \mid D, E \in \tilde{\mcE}, \ D \subseteq E \} \bigr), \\
 H^{\ti{\mcE}}& := \bigl( \{ H(D) \mid D   \in \tilde{\mcE}\} , \{ H(\mcI^{DE}) \mid D, E \in \tilde{\mcE}, \ D \subseteq E \} \bigr)
\end{align*}
 with (partially ordered) index set $\tilde{\mcE}$ are direct systems of graded $\Z$-modules in the category of $\Z$-moduls. In particular, both are direct systems of abelian groups and therefore admit direct limits.
 If the chain compatible set $\tilde{\mcE}\subset (\overline{\mcE}, \subseteq)$ is not only partially ordered but in fact directed then both direct systems allow passing to the direct limit of short exact sequences and we can assign a welldefined homology via direct limit to $\ti{\mcE}$:
$$
H_k(\ti{\mcE}):= \varinjlim H_k^{\ti{\mcE}} \quad \forall\ k \in \Z \qquad \mbox{and} \qquad H(\ti{\mcE}):= \bigoplus_{k \in \Z} H_k(\ti{\mcE}).
$$
\end{theorem}

\begin{proof}
Let $\ti{\mcE} \subset (\overline{\mcE}, \subseteq)$ be chain compatible and consider $D, E \in \ti{\mcE}$ with $D \subseteq E$.
Then, by definition, the map $\mcI^{DE}: (\Z D, \del^D) \to (\Z E, \del^E)$ is a chain map, i.e., $\del^E_k \circ \mcI^{DE}_k =   \mcI^{DE}_{k-1} \circ \del^D_k$ for all $k \in \Z$.
This implies in particular that, for all $k \in \Z$, its restriction to the images $\Img \del^D_k$ and kernels $\ker \del^D_k$ satisfies
\begin{equation*}
  \mcI_{k-1}^{DE}|_{\Img \del^D_k}  : \Img \del^D_k \to \Img \del^E_k
 \qquad \mbox{and} \qquad
 \mcI_k^{DE}|_{\ker \del^D_k}  : \ker \del^D_k \to \ker \del^E_k.
\end{equation*}
Therefore
\begin{align*}
 &\ker_k^{\ti{\mcE}}:= (\ker_k, \mcI_k)  := \bigl( \{ \ker \del_k^E \mid E \in \ti{\mcE} \}, \{\mcI^{DE}_k \mid D, E \in \ti{\mcE}, D \subseteq E \} \bigr),  \\
 &\Img_k^{\ti{\mcE}}:= (\Img_k, \mcI_k)  := \bigl( \{ \Img \del_k^E \mid E \in \ti{\mcE} \}, \{\mcI^{DE}_k \mid D, E \in \ti{\mcE}, D \subseteq E \} \bigr)
\end{align*}
are, for all $k \in \Z$, direct systems over the partially ordered, chain compatible index set $\ti{\mcE}$. Since $\ti{\mcE}$ is chain compatible, $H_k(E)= \ker \del_k^E \slash \Img \del_{k+1}^E $, the $k$th (local) homoclinic Floer group of $E \in \ti{\mcE}$ (see \refnewDefLocalFH), is welldefined. Moreover, according to \refPropHomSys,
$$
H_k^{\ti{\mcE}}:=  \bigl( \{ H_k(E)= \bigl(\ker \del_k^E \slash \Img \del_{k+1}^E \bigr) \mid E \in \ti{\mcE} \}, \{H_k(\mcI^{DE}) \mid D, E \in \ti{\mcE}, D \subseteq E \} \bigr)
$$
with
$$H_k(\mcI^{DE}) \ : \ H_k(D)= \ker \del_k^D \slash \Img \del_{k+1}^D  \ \longrightarrow \  H_k(E)=\ker \del_k^E \slash \Img \del_{k+1}^E$$
is a direct system.
Note that, for all $k \in \Z$ and all $E \in \ti{\mcE}$, we have short exact sequences
$$
0 \to \Img \del^E_{k+1} \hookrightarrow \ker \del^E_k \twoheadrightarrow  \ker \del^E_k \slash  \Img \del^E_{k+1} = H_k(E) \to 0
$$
that give rise to two short exact sequences of direct systems:
$$
0 \to \Img_{k+1}^{\ti{\mcE}} \hookrightarrow \ker_k^{\ti{\mcE}} \twoheadrightarrow  \ker_k^{\ti{\mcE}} \slash  \Img _{k+1}^{\ti{\mcE}} \to 0
\quad \mbox{and} \quad
0 \to \Img_{k+1}^{\ti{\mcE}} \hookrightarrow \ker_k^{\ti{\mcE}} \twoheadrightarrow  H_k^{\ti{\mcE}} \to 0
$$
These two short exact sequences coincide if $\ti{\mcE}$ is not only partially ordered but in fact directed since then we have according to \refshortDirectLim\ and \refdirLimQuotient
\begin{align*}
\varinjlim H_k^{\ti{\mcE}} & = \varinjlim H_k (E)= \varinjlim \ \bigl( \ker \del_k^E \slash \Img \del_{k+1}^E \bigr) \\
& = \left. \bigl( \varinjlim \ker \del_k^E \bigr) \ \right/ \bigl( \varinjlim \Img \del_{k+1}^E \bigr)  = \left. \varinjlim \ker_k \right/ \varinjlim \Img_k
\end{align*}
Thus, to a directed and chain compatible set $\ti{\mcE}\subset \overline{\mcE}$, we can assign a welldefined homology via taking the direct limit:
$$
H_k(\ti{\mcE}):= \varinjlim H_k^{\ti{\mcE}} \quad \forall\ k \in \Z \qquad \mbox{and} \qquad H(\ti{\mcE}):= \bigoplus_{k \in \Z} H_k(\ti{\mcE}).
$$
\end{proof}

\begin{remark}
In the huge set $\overline{\mcE}$, there are --- or probably rather must be --- many different directed, chain compatible subsets $\ti{\mcE} \subset \overline{\mcE}$ since the complexity of the underlying homoclinic tangle cannot be caught by `just one (or maybe several) direct limits' which are in essence `quite finitely generated objects' due to their nature as direct sum.
So there is still a lot of structure to be detected and studied!
\end{remark}


\subsection{Comments on inverse limits}

In Section \ref{sec:addCutPartners}, we saw that the attempt to `add cutting points to make a $\del$-incomplete set $\del$-complete' does not (yet) work, but, in Section \ref{sec:pruning}, that the pruning algorithm works well. Pruning all cutting points that do not have cutting partners means that the set in question becomes potentially smaller. Unfortunately this does not fit well together with direct limits since the partial order of the underlying index set maps given sets into larger sets.

\vsp

Thus one may wonder about the dual procedure, i.e., working with restrictions/projections instead of inclusions and then taking the inverse limit, i.e., one would replace the inclusions $\mcI^{DE}: \Z D \to \Z E$ associated with $D, E \in \mcE$ satisfying $D \subseteq E$ by restrictions $\mcR^{E' D'}: \Z E' \to \Z D' $ associated with $E', D' \in \mcE$ with $E' \supseteq D'$. More explicitly, define the Kronecker symbol on $\mcH$ by
$$
\langle \cdot, \cdot \rangle : \mcH \x \mcH \to \{0, 1\}, \qquad \langle p, p' \rangle :=
\left\{
\begin{aligned}
1, & \quad \mbox{ if } p = p', \\
0, & \quad \mbox{ if } p \neq p'
\end{aligned}
\right.
$$
and, for $E', D' \in \mcE$ with $E' \supseteq D'$, define the restriction $\mcR^{E'D'}: \Z E' \to \Z D'$ on the generators  by
$$
\mcR^{E'D'}(p):= p-\sum_{\stackrel{p' \in E' \setminus D'}{\mu(p' ) = \mu(p)}} \langle p, p' \rangle \ p' =
\left\{ 
\begin{aligned}
p, & \quad \mbox{ if } p \in D', \\
0, & \quad \mbox{ if } p \in E' \setminus D',
\end{aligned}
\right. 
$$
and extend it by linearity. A calculation shows that, for all $D', E', F' \in \mcE$ with $F' \supseteq E' \supseteq D'$, we have $\mcR^{F'E'} \circ \mcR^{E'D'} = \mcR^{F' D'}$.
Now recall the splitting of $\del^{E'} = \del^{E'_{D'}} + \del ^{E'_{E' \setminus D'}}$ from \refcriterionChainCompat. By definition, $\mcR^{E'D'}: \Z E' \to \Z D'$ is a chain map if
\begin{equation}
\label{RchainMap}
 \mcR^{E'D'}_{k-1} \circ \del_k^{E'} = \del^{D'}_k \circ  \mcR^{E'D'}_{k}.
\end{equation}
 So, when is this the case?

\begin{lemma}[Criterion]
Let $E', D' \in \overline{\mcE}$ with $E' \supseteq D'$. Then
$$
\mcR^{E'D'}: (\Z E', \del^{E'}) \ \to \ (\Z D', \del ^{D'}) \ \mbox{ is a chain map}
\quad \IFF\quad
\del^{E'_{D'}} \mbox{ vanishes on } E' \setminus D'.
$$
\end{lemma}

\begin{proof}
Let $E', D' \in \overline{\mcE}$ with $E' \supseteq D'$ and consider the disjoint union $E' = D' \cup (E' \setminus D')$. When evaluating Equation \eqref{RchainMap} on $p' \in D' $ seen as point in $ E'$ we find for all $k \in \Z$
$$
\del^{D'}_k \circ \mcR_k^{E' D'}|_{D'} = \mcR_{k-1}^{E' D'} \circ \del^{E'}|_{D'} ,
$$
i.e., here the chain map property is always true.
Now we do the same with a point $p' \in E' \setminus D' $ and find for all $k \in \Z$
$$
\del^{D'}_k \circ \mcR_k^{E' D'}|_{E' \setminus D'} =0
\quad \mbox{and} \quad
 \mcR_{k-1}^{E' D'} \circ \del^{E'}|_{E' \setminus D'} = \mcR_{k-1}^{E' D'} \circ \del^{E'_{D'}}|_{E' \setminus D'}.
$$
Thus, for $\mcR^{E' D'}$ to be a chain map, we need $\mcR_{k-1}^{E' D'} \circ \del^{E'_{D'}}|_{E' \setminus D'}=0$ which is equivalent with requiring $\del^{E'_{D'}}|_{E' \setminus D'}=0$, i.e., $\del^{E'_{D'}}$ must vanish on $E' \setminus D'$.
\end{proof}

So the framework of a inverse limit can be set up analogously to the direct limit. But since the partially ordered set underlying an inverse limit proceeds by restrictions, i.e., the sets `become smaller' along the way, this is counterintuitive to what we actually want: We want to enlarge the sets and accumulate more information, not make them smaller. Thus, for this geometric-dynamical reason we did not further pursue the approach via inverse limits.






{\small
  \noindent
  \\
  Sonja Hohloch\\
  University of Antwerp\\
  Department of Mathematics\\
  Middelheimlaan 1\\
  B-2020 Antwerpen, Belgium\\
  {\em E\--mail}: \texttt{sonja DOT hohloch AT uantwerpen DOT be}

\end{document}